\documentclass[11pt,centertags,a4paper,reqno]{amsart}

\usepackage{amsmath}
\usepackage{amsthm}
\usepackage{amssymb}
\usepackage{enumitem}
\usepackage{esint}
\usepackage{upgreek}
\usepackage{hyperref}
\usepackage{bm}
\usepackage{mathrsfs}
\usepackage{tikz-cd}

\numberwithin{equation}{section}
\numberwithin{figure}{section}

\allowdisplaybreaks

\theoremstyle{definition}
	\newtheorem{definition}{Definition}[section]
	\numberwithin{definition}{section}

\theoremstyle{plain}
	\newtheorem{lemma}[definition]{Lemma}
	\newtheorem{proposition}[definition]{Proposition}
	\newtheorem{theorem}{Theorem}
	\newtheorem*{conjecture*}{Conjecture}
	\newtheorem{corollary}[definition]{Corollary}
	\newtheorem*{corollary*}{Corollary}

\let \L \relax
\let \d \relax
\let \Re \relax
\let \Im \relax
\let \d \relax
\let \i \relax
\let \P \relax

\DeclareMathOperator{\I}{I}
\DeclareMathOperator{\E}{E}

\DeclareMathOperator{\Q}{Q}

\DeclareMathOperator{\de}{\delta}
\DeclareMathOperator{\Id}{Id}
\DeclareMathOperator{\L}{L}
\DeclareMathOperator{\Lbar}{\overline{L}}
\DeclareMathOperator{\Re}{Re}
\DeclareMathOperator{\Im}{Im}

\DeclareMathOperator{\P}{P}
\DeclareMathOperator{\Rop}{R}
\DeclareMathOperator{\Sop}{S}
\DeclareMathOperator{\piop}{\pi}
\DeclareMathOperator{\sigmaop}{\sigma}
\DeclareMathOperator{\im}{im}

\renewcommand{\th}{\theta}
\newcommand{\As}{\mathscr{A}}
\newcommand{\Ls}{\mathscr{L}}
\newcommand{\Ss}{\mathscr{S}}
\newcommand{\Ac}{\mathcal{A}}
\newcommand{\Bc}{\mathcal{B}}
\newcommand{\Ec}{\mathcal{E}}
\newcommand{\C}{\mathbb{C}}
\newcommand{\Z}{\mathbb{Z}}

\newcommand{\R}{\mathbb{R}}
\newcommand{\T}{\mathbb{T}}
\newcommand{\To}{\mathring{\mathbb{T}}}
\newcommand{\K}{\C}
\newcommand{\Dbb}{\mathbb{D}}
\newcommand{\map}[3]{#1\colon#2\rightarrow#3}
\newcommand{\deq}{\mathrel{\mathop:}=}
\newcommand{\dd}[2]{\frac{\mathrm{d} #1}{\mathrm{d} #2}}
\newcommand{\pd}[2]{\frac{\partial #1}{\partial #2}}

\newcommand{\vsharp}{v^{\sharp}}
\newcommand{\vflat}{v^{\flat}}
\newcommand{\usharp}{u^{\sharp}}
\newcommand{\uflat}{u^{\flat}}

\newcommand{\SL}{\mathrm{SL}}
\newcommand{\PSL}{\mathrm{PSL}}
\newcommand{\PU}{\mathrm{PU}}
\newcommand{\cb}{\mathrm{cb}}
\renewcommand{\c}{\mathrm{c}}
\renewcommand{\b}{\mathrm{b}}
\renewcommand{\t}{\mathrm{\tau}}

\newcommand{\d}{\mathrm{d}}
\newcommand{\i}{i}
\newcommand{\ftilde}{\tilde{f}}
\newcommand{\z}{\mathbf{z}}
\newcommand{\thul}{\bm{\uptheta}}
\newcommand{\Or}{\mathrm{or}}
\newcommand{\m}{\mu}

\begin{document}

\title{Transgression in bounded cohomology and a conjecture of Monod}

\author[A. Ott]{Andreas Ott}
\address{Mathematisches Institut, Ruprecht-Karls-Universit\"at Heidelberg, Mathematikon, Im Neuenheimer Feld 205, 69120 Heidelberg, Germany}
\email{aott@mathi.uni-heidelberg.de}

\begin{abstract}
	We develop an algebro-analytic framework for the systematic study of the continuous bounded cohomology of Lie groups in large degree. As an application, we examine the continuous bounded cohomology of $\PSL(2,\R)$ with trivial real coefficients in all degrees greater than two. We prove a vanishing result for strongly reducible classes, thus providing further evidence for a conjecture of Monod. On the cochain level, our method yields explicit formulas for cohomological primitives of arbitrary bounded cocycles.
\end{abstract}

\maketitle

\tableofcontents

\section{Introduction and statement of the results}

Bounded cohomology of discrete groups was introduced into geometry by Gromov \cite{Gromov/Volume-and-bounded-cohomology}. The theory was subsequently extended to locally compact second countable groups by Burger and Monod \cite{Burger/Bounded-cohomology-of-lattices-in-higher-rank-Lie-groups,Burger/Continuous-bounded-cohomology-and-applications-to-rigidity-theory,Monod/Continuous-bounded-cohomology-of-locally-compact-groups}, who coined the term \emph{continuous bounded cohomology}. Bounded cohomology has by now proved itself an indispensable tool in geometry, topology and group theory, see for example the references surveyed in \cite{Hartnick/Bounded-cohomology-via-partial-differential-equations-I}. Nevertheless, the structure of the bounded cohomology ring of a given group is in general not very well understood. Existing results are chiefly concerned with bounded cohomology in low degrees, most notably bounded cohomology in degree $2$, which is intimately linked with quasi-morphisms (e.g.~\cite{Brooks/Some-remarks-on-bounded-cohomology,Grigorchuk/Some-results-on-bounded-cohomology,Epstein/The-second-bounded-cohomology-of-word-hyperbolic-groups,Bestvina/Bounded-cohomology-of-subgroups-of-mapping-class-groups,Burger/Bounded-cohomology-of-lattices-in-higher-rank-Lie-groups,Burger/Continuous-bounded-cohomology-and-applications-to-rigidity-theory,BestvinaBrombergFujiwara/Bounded-cohomology-with-coefficients-in-uniformly-convex-Banach-spaces,HullOsin/Induced-quasicocycles-on-groups-with-hyperbolically-embedded-subgroups,AntolinMjSistoTaylor/Intersection-properties-of-stable-subgroups-and-bounded-cohomology,HartnickSisto/Bounded-cohomology-and-virtually-free-hyperbolically-embedded-subgroups}), and bounded cohomology in degree $3$, which has close ties with the geometry of $3$-manifolds (e.g.~\cite{Brooks/Some-remarks-on-bounded-cohomology,Soma/Bounded-cohomology-and-topologically-tame-Kleinian-groups,Soma/Bounded-cohomology-of-closed-surfaces,Soma/The-zero-norm-subspace-of-bounded-cohomology,Farre/Bounded-cohomology-of-finitely-generated-Kleinian-groups,Farre/Relations-in-bounded-cohomology,Burger/On-and-around-the-bounded-cohomology-of-SL2,Pieters/Continuous-cohomology-of-the-isometry-group-of-hyperbolic-space-realizable-on-the-boundary,FranceschiniFrigerioPozzettiSisto/The-zero-norm-subspace-of-bounded-cohomology-of-acylindrically-hyperbolic-groups}). Bounded cohomology in higher degrees, on the contrary, is still largely unexplored. There is a number of known bounded cohomology classes in higher degree, often emerging from explicit geometric constructions (e.g.~ \cite{Dupont/Bounds-for-characteristic-numbers-of-flat-bundles,Thurston/Three-dimensional-geometry-and-topology.-Vol.-1,Gromov/Volume-and-bounded-cohomology,Goncharov/Geometry-of-configurations-polylogarithms-and-motivic-cohomology,Mineyev/Bounded-cohomology-characterizes-hyperbolic-groups,Bucher-Karlsson/Finiteness-properties-of-characteristic-classes-of-flat-bundles,Lafont/Simplicial-volume-of-closed-locally-symmetric-spaces-of-non-compact-type,Bucher-Karlsson/Simplicial-volume-of-locally-symmetric-spaces-covered-by-rm-SLsb-3Bbb-R/rm-SO3,Bucher/The-norm-of-the-Euler-class,Hartnick/Surjectivity-of-the-comparison-map-in-bounded-cohomology-for-Hermitian-Lie-groups,Hartnick/Bounded-cohomology-via-partial-differential-equations-I,BucherMonod/The-cup-product-of-Brooks-quasimorphisms,Heuer/Cup-Product-in-Bounded-Cohomology-of-the-Free-Group}). On the other hand, a classical result due to Johnson \cite{Johnson/Cohomology-in-Banach-algebras} asserts that the bounded cohomology of an amenable group vanishes in all positive degrees. Moreover, L\"oh \cite{Loh/A-note-on-bounded-cohomological-dimension-of-discrete-groups} recently found non-amenable groups whose bounded cohomology with trivial real coefficients vanishes in all positive degrees, and Bucher and Monod \cite{BucherMonod/The-bounded-cohomology-of-SL2-over-local-fields-and-S-integers} proved a similar statement for $\SL_{2}$ over non-Archimedian local fields. These latter results have in common that the bounded cohomological dimension of the respective group is zero. In fact, it is presently not known if there exists any group with non-zero finite bounded cohomological dimension. In a different direction, Monod \cite{Monod/On-the-bounded-cohomology-of-semi-simple-groups-S-arithmetic-groups-and-products,Monod/Vanishing-up-to-the-rank-in-bounded-cohomology} proved vanishing in degree below twice the rank for the bounded cohomology of non-amenable semisimple groups with non-trivial coefficients.  

\medskip

Our goal in this article is to initiate a systematic study of bounded cohomology in large degree. In view of the following conjecture of Monod
it is natural to focus attention, for the time being, on the continuous bounded cohomology of Lie groups with trivial real coefficients. The conjecture also suggests what the precise meaning of large degree should be in this case, as we will readily see.

\begin{conjecture*}[Monod \cite{Monod/An-invitation-to-bounded-cohomology}]
	Let G be a connected semisimple Lie group with finite center.	Then the natural comparison map $H^\bullet_{\cb}(G;\R) \to H^\bullet_{\c}(G;\R)$ from the continuous bounded cohomology to the continuous cohomology of $G$ is an isomorphism in all degrees.
\end{conjecture*}

Surjectivity of the comparison map was already studied by Dupont \cite{Dupont/Bounds-for-characteristic-numbers-of-flat-bundles} and has since been established in many cases \cite{Dupont/Bounds-for-characteristic-numbers-of-flat-bundles,Thurston/Three-dimensional-geometry-and-topology.-Vol.-1,Gromov/Volume-and-bounded-cohomology,Goncharov/Geometry-of-configurations-polylogarithms-and-motivic-cohomology,Bucher-Karlsson/Finiteness-properties-of-characteristic-classes-of-flat-bundles,Lafont/Simplicial-volume-of-closed-locally-symmetric-spaces-of-non-compact-type,Bucher-Karlsson/Simplicial-volume-of-locally-symmetric-spaces-covered-by-rm-SLsb-3Bbb-R/rm-SO3,Hartnick/Surjectivity-of-the-comparison-map-in-bounded-cohomology-for-Hermitian-Lie-groups}, while still almost nothing is known about injectivity. In fact, injectivity of the comparison map has so far only been proved in degree $2$ by Burger and Monod \cite{Burger/Bounded-cohomology-of-lattices-in-higher-rank-Lie-groups}, in degree $3$ for certain groups of rank $1$ by Burger and Monod \cite{Burger/On-and-around-the-bounded-cohomology-of-SL2}, Bloch \cite{Bloch/Higher-regulators-algebraic-K-theory-and-zeta-functions-of-elliptic-curves} and Pieters \cite{Pieters/Continuous-cohomology-of-the-isometry-group-of-hyperbolic-space-realizable-on-the-boundary}, and in degree $4$ for $\SL_{2}(\R)$ by Hartnick and the author \cite{Hartnick/Bounded-cohomology-via-partial-differential-equations-I}. The conjecture predicts that the bounded cohomological dimension of a connected semisimple Lie group $G$ with finite center equals the dimension of the symmetric space associated to  $G$, and is hence positive and finite. In particular, we expect that $H^{n}_{\cb}(G;\R) = 0$ whenever the degree $n$ is sufficiently large in the sense that it exceeds the dimension of the symmetric space of $G$.

\medskip

The present article is devoted to the examination of this sort of conjectural vanishing of continuous bounded cohomology in large degree. We will always assume that $G$ is a connected real Lie group that is locally isomorphic to $\PSL_{2}(\R)$. Note that in this case, Monod's conjecture predicts that $H^{n}_{\cb}(G;\R) = 0$ for all $n>2$. Theorem~\ref{thm:VanishingForStronglyReducibleClasses} below shows that the conjecture holds for all classes in degree $n>2$ that are strongly reducible in the following sense. A bounded cohomology class $\alpha \in H^{n}_{\cb}(G;\R)$ is called \emph{strongly reducible} if it admits a product decomposition
\[
\alpha = \alpha^{\prime} \smallsmile \alpha^{\prime\prime}
\]
with factors $\alpha^{\prime} \in H^{2}_{\cb}(G;\R)$ and $\alpha^{\prime\prime} \in H^{n-2}_{\cb}(G;\R)$. Here we denote by $\smallsmile$ the natural cup product on the continuous bounded cohomology of $G$ (see Section \ref{subsec:ContinuousBoundedCohomology}). We are going to prove the following vanishing theorem for strongly reducible classes.

\begin{theorem} \label{thm:VanishingForStronglyReducibleClasses}
	Let $G$ be a connected real Lie group that is locally isomorphic to $\PSL_{2}(\R)$, and consider a class $\alpha \in H^{n}_{\cb}(G;\R)$ of degree $n>2$ in the continuous bounded cohomology of $G$ with trivial real coefficients. If $\alpha$ is strongly reducible, then $\alpha = 0$.
\end{theorem}

Thinking of $G$ as the Hermitian Lie group $\PU(1,1)$, we may also regard Theorem \ref{thm:VanishingForStronglyReducibleClasses} from the following different perspective. Recall that by a result of Burger and Monod, in this case the second continuous bounded cohomology $H^{2}_{\cb}(G;\R)$ is generated by the bounded K\"ahler class $\kappa$ (see Section \ref{subsec:ContinuousBoundedCohomology}). We then consider the \emph{bounded Lefschetz map}
\begin{equation} \label{map:BoundedLefschetzMap} \tag{1}
	\map{L_{\kappa}^{\bullet}}{H^{\bullet}_{\cb}(G;\R)}{H^{\bullet+2}_{\cb}(G;\R)}
\end{equation}
defined by $L_{\kappa}(\alpha) = \kappa \smallsmile \alpha$.

\begin{corollary*}
	The bounded Lefschetz map in \eqref{map:BoundedLefschetzMap} is zero in all positive degrees.
\end{corollary*}

Returning to Theorem \ref{thm:VanishingForStronglyReducibleClasses}, we note that in small degrees $n=3,4$, much stronger vanishing theorems apply: Burger and Monod \cite{Burger/On-and-around-the-bounded-cohomology-of-SL2} proved that $H^{3}_{\cb}(G;\R) = 0$, while Hartnick and the author \cite{Hartnick/Bounded-cohomology-via-partial-differential-equations-I} showed that $H^{4}_{\cb}(G;\R) = 0$. In large degree, on the other hand, our Theorem \ref{thm:VanishingForStronglyReducibleClasses} establishes the first non-trivial vanishing result for classes in $H^{n}_{\cb}(G;\R)$ in arbitrary degree $n>4$. The proofs of all these vanishing results crucially rely on the boundary resolution for continuous bounded cohomology due to Ivanov \cite{Ivanov/Foundations-of-the-theory-of-bounded-cohomology} and Burger and Monod \cite{Burger/Bounded-cohomology-of-lattices-in-higher-rank-Lie-groups}. In fact, in this particular resolution all cocycles vanish in degree $n=3$. In degree $n>3$, this is no longer the case and one faces the problem of constructing bounded primitives. This was accomplished in degree $n=4$ by Hartnick and the author \cite{Hartnick/Bounded-cohomology-via-partial-differential-equations-I} by means of a new technique that employs differential equations in order to explicitly construct bounded primitives; the arguments, however, crucially rely on the assumption that $n$ be sufficiently small.

\medskip

In this article, we will take the ideas from \cite{Hartnick/Bounded-cohomology-via-partial-differential-equations-I} further and develop an algebro-analytic framework that allows to overcome any upper bounds on the degree in constructing bounded primitives by means of differential equations. At the heart of this approach lies the \emph{transgression map}
\begin{equation} \label{eqn:IntroductionTransgressionMap} \tag{2}
	\map{\Lambda^{n}}{H^{n-2}(\Ac^{\infty})}{H^{n}_{\cb}(G;\R}) \quad\quad (n>2)
\end{equation}
from the shifted cohomology of a certain cochain complex $\Ac^{\infty}$ to the continuous bounded cohomology of $G$ (see Section \ref{subsec:TransgressionMap}). Its construction is the main theme of this work. Notice that the transgression map is defined in every degree $n>2$. Theorem~\ref{thm:VanishingForStronglyReducibleClasses} is then a consequence of the next theorem, which clarifies how transgression gives rise to the vanishing of strongly reducible bounded cohomology classes.

\begin{theorem} \label{thm:ImageOfTransgressionMap}
	For every $n>2$, the transgression map in \eqref{eqn:IntroductionTransgressionMap} has the following properties.
\begin{enumerate}[leftmargin=1cm,topsep=0.5ex,itemsep=0.5ex]
	\item The cochain complex $\Ac^{\infty}$ is acyclic, and hence all elements in the image of $\Lambda^{n}$ necessarily vanish.
	\item Strongly reducible classes in $H^{n}_{\cb}(G;\R)$ are contained in the image of $\Lambda^{n}$.
\end{enumerate}
\end{theorem}

We will refer to elements in the image of the transgression map as \emph{transgressive} classes. The main ingredient of our proof of Theorem \ref{thm:ImageOfTransgressionMap} is then a cohomological characterization of transgressive classes, see Proposition \ref{prop:CriterionForTransgressive} in Section \ref{subsec:TransgressiveClasses}. Let us note in passing that in view of Monod's conjecture, it appears natural to speculate that the transgression map $\Lambda^{n}$ in \eqref{eqn:IntroductionTransgressionMap} is in fact surjective for every $n>2$.

\medskip

A particular feature of our approach is that it yields explicit formulas for primitives of bounded cocycles. To make this precise, let us assume that $G = \PU(1,1)$ and recall that for all $n\ge0$ the boundary model of Burger and Monod gives rise to an isomorphism
\[
H^{n}_{\cb}(G;\R) \cong H^{n}( L^{\infty}(\T^{\bullet+1},\R)^{G},\de^{\bullet} )
\]
between the continuous bounded cohomology of $G$ and the cohomology of the homogeneous cochain complex
\[
\begin{tikzcd}[cramped]
	0 \arrow[r,rightarrow]
	& L^{\infty}(\T^{1},\R)^{G} \arrow[r,"\de^{0}",rightarrow]
	& L^{\infty}(\T^{2},\R)^{G} \arrow[r,"\de^{1}",rightarrow]
	& L^{\infty}(\T^{3},\R)^{G} \arrow[r,"\de^{2}",rightarrow]
	& \cdots
\end{tikzcd}
\]
of $G$-invariant bounded functions defined on the Furstenberg boundary of the Lie group $G$ (see Section \ref{subsec:BoundaryModel}). In this way, any class $\alpha \in H^{n}_{\cb}(G;\R)$ is identified with the cohomology class $[c]$ of some $G$-invariant bounded cocycle $c \in L^{\infty}(\T^{n+1},\R)^{G}$. We see that $\alpha$ vanishes if and only if the cocycle $c$ admits a $G$-invariant bounded primitive $p \in L^{\infty}(\T^{n},\R)^{G}$ that satisfies the cohomological equation
\begin{equation} \label{eqn:IntroductionPrimitive} \tag{3}
	\de p = c.
\end{equation}
Explicit solutions of this equation, as well as the equation itself, are often closely related with classical transcendental functions and their functional equations. For example, in degree $n=4$ there is an intricate connection with Euler's dilogarithm function and the Spence-Abel functional equation \cite{Bloch/Higher-regulators-algebraic-K-theory-and-zeta-functions-of-elliptic-curves,Burger/On-and-around-the-bounded-cohomology-of-SL2,HartnickOtt/Perturbations-of-the-Spence-Abel-equation-and-deformations-of-the-dilogarithm-function}.

Our next result, which is Theorem \ref{thm:ExplicitFormulasForPrimitives} below, systematically constructs measurable solutions of \eqref{eqn:IntroductionPrimitive} in all degrees $n>2$ by means of certain explicit line integrals, and provides a sufficient criterion for their boundedness. We denote by $L^{0}(\T^{n},\R)^{G}$ the space of $G$-invariant measurable functions (see Section \ref{subsec:BoundedMeasurableFunctions}), by $\Or \in L^{\infty}(\T^{3},\R)^{G}$ the orientation cocycle, and by $\cup$ the natural cup product for cochains on the boundary of $G$ (see Section \ref{subsec:ReducibleClasses}).

\begin{theorem} \label{thm:ExplicitFormulasForPrimitives}
	There exists a linear operator
\[
\map{\P^{n}}{L^{\infty}(\T^{n+1},\R)^{G} \supset \ker \de^{n}}{L^{0}(\T^{n},\R)^{G}} \quad\quad (n>2)
\]
with the following properties.
\begin{enumerate}[leftmargin=1cm,topsep=0.5ex,itemsep=0.5ex]
	\item The operator $\P^{n}$ is defined by Formula \eqref{eqn:DefinitionOfOperatorP} in Section \ref{subsec:TheOperatorP}.
	\item For every $n>2$, and for every $G$-invariant bounded function $c \in L^{\infty}(\T^{n+1},\R)^{G}$ satisfying the cocycle relation $\de c = 0$, the function $\P c \in L^{0}(\T^{n},\R)^{G}$ is a $G$-invariant primitive for $c$ that solves the cohomological equation
\[
\de \P c = c.
\]
	\item Assume in addition that the cocycle $c$ admits a product decomposition
\begin{equation} \label{eqn:ProductDecompositionForCocycle} \tag{4}
	c = \Or \cup c^{\prime}
\end{equation}
for some cocycle $c^{\prime} \in L^{\infty}(\T^{n-1},\R)^{G}$, where $\Or \in L^{\infty}(\T^{3},\R)^{G}$ denotes the orientation cocycle (see Section \ref{subsec:ReducibleClasses}). Then the primitive $\P c$ is bounded.
\end{enumerate}
\end{theorem}

We remark that in those cases in which the cocycle $c$ does not admit a product decomposition as in \eqref{eqn:ProductDecompositionForCocycle}, it is presently not known whether the solution $p = \P c$ of \eqref{eqn:IntroductionPrimitive} provided by Theorem~\ref{thm:ExplicitFormulasForPrimitives} is bounded or not.

\medskip

The article is organized as follows. In Section \ref{sec:FunctionSpaces}, we fix some notation and terminology, and define several function spaces that will later be used when working with differential equations in the context of bounded cohomology. Section \ref{sec:Cohomology} collects basic facts about the continuous bounded cohomology of Lie groups, and studies the cohomological properties of the function spaces defined in the previous section. In Section \ref{sec:CauchyFrobeniusComplex}, we introduce the Cauchy-Frobenius differential complex. We construct solutions of the corresponding differential equations and study their boundedness properties. In Section \ref{sec:Transgression}, we combine the cohomological results from Section \ref{sec:Cohomology} with the analytic results from Section \ref{sec:CauchyFrobeniusComplex} in order to define the transgression map in \eqref{eqn:IntroductionTransgressionMap}. We investigate strongly reducible bounded cohomology classes and prove Theorem \ref{thm:VanishingForStronglyReducibleClasses} and Theorem \ref{thm:ImageOfTransgressionMap}. The final Section \ref{sec:ConstructionOfPrimitives} is concerned with the explicit construction of solutions for \eqref{eqn:IntroductionPrimitive}, leading to a proof of Theorem \ref{thm:ExplicitFormulasForPrimitives}.

\bigskip

\noindent \textbf{Acknowledgements.} \!The author wishes to thank Y.\,Benoist, M.\,Burger, O.\,Forster, T.\,Hart-nick, G.\,Kings, R.\,Mazzeo, M.\,Puschnigg, J.\,Schmidt, J.\,Swoboda, R.\,Weissauer, A.\,Wienhard, F.\,Ziltener, and M.\,Zworski for helpful discussions and suggestions, as well as I.\,Agol and the UC Berkeley Mathematics Department for their hospitality and excellent working conditions. He was supported by the European Research Council under ERC-Consolidator Grant 614733 ``Deformation Spaces of Geometric Structures'', and by the Priority Program 2026 ``Geometry at Infinity'' of the German Research Foundation under DFG grant 340014991. The author further acknowledges support from U.S.\,National Science Foundation grants DMS 1107452, 1107263, 1107367 ``RNMS:\,GEometric structures And Representation varieties'' (the GEAR Network).

\section{Function spaces}
\label{sec:FunctionSpaces}

\subsection{The Lie group $G$}
\label{subsec:LieGroupG}

Let us fix the Lie group $G \deq \PU(1,1)$. Elements of this Lie group are represented by matrices of the form
\[
g_{a,b} \deq \begin{pmatrix} a & b \\ \overline{b} & \overline{a} \end{pmatrix},
\]
with complex numbers $a,b \in \C$ satisfying $|a|^{2}-|b|^{2} = 1$. Denote by $[g_{a,b}]$ the equivalence class of the matrix $g_{a,b}$ in $G$. In particular, for $t \in \R$ we fix the notation
\[
k_{t} \deq \big[ g_{e^{\i t/2}, \, 0} \big], \quad a_{t} \deq \big[ g_{\cosh(-t/2), \, \sinh(-t/2)} \big], \quad n_{t} \deq \big[ g_{1+\frac{1}{2}\i t, \, -\frac{1}{2}\i t} \big].
\]
Note that these elements are elliptic, hyperbolic and parabolic, respectively. They give rise to Lie subgroups
\[
K \deq \{ k_{t} \,|\, t \in \R \}, \quad A \deq \{ a_{t} \,|\, t \in \R \}, \quad N \deq \{ n_{t} \,|\, t \in \R \},
\]
which are $1$-parameter subgroups in the sense that the maps $t \mapsto k_{t}$, $t \mapsto a_{t}$ and $t \mapsto n_{t}$ are smooth homomorphisms $\R \to G$. The group $K$ is a maximal compact subgroup of $G$. It is isomorphic with the unit circle $S^{1}$ via the identification $k_{t} \mapsto e^{\i\,t}$. For later reference, we note that $A$ normalizes $N$, and in particular, that there is a relation
\begin{equation} \label{eqn:ANormalizesN}
	a_{s}.n_{t}.a_{s}^{-1} = n_{e^{-s} \cdot t}
\end{equation}
for $s,t \in \R$. The product $P \deq AN$ is a parabolic subgroup of $G$. Moreover, every elliptic, hyperbolic or parabolic element in $G$ is conjugate to an element in the subgroup $K$, $A$ or $N$, respectively. In this way, we obtain the Iwasawa decomposition $G = KAN$ and a Cartan decomposition $G = KAK$. Note that the Iwasawa decomposition is unique, while the Cartan decomposition is not. We will write a Cartan decomposition for any $g \in G$ in the form
\begin{equation} \label{eqn:CartanDecomposition}
	g = k^{\prime} \, a_{t} \, k
\end{equation}
with elements $k, k^{\prime} \in K$ and $a_{t} \in A$ for some $t \in \R$. For more details see \cite[Ch.\,VI]{Knapp/Lie-groups-beyond-an-introduction} and \cite[Ch.\,V]{Sugiura/Unitary-representations-and-harmonic-analysis}.

\subsection{Boundary action}
\label{subsec:BoundaryAction}

$G$ acts smoothly on the closed unit disk $\overline{\Dbb} = \{ z \in \C \,|\, |z| \le 1 \}$ by fractional linear transformations. This action restricts to a smooth $G$-action $G \times S^{1} \to S^{1}$ on the unit circle $S^{1} \subset \C$, denoted by $(g,z) \mapsto g.z$. Thinking of $S^{1}$ as the Furstenberg boundary of $G$, we will refer to this action as the \emph{boundary action} of $G$. Recall that the boundary action is strictly $3$-transitive \cite[Thm.\,11.1]{Kerby/On-infinite-sharply-multiply-transitive-groups} and amenable \cite[Prop.\,4.3.2]{Zimmer/Ergodic-theory-and-semisimple-groups}. The induced action of the maximal compact subgroup $K$ is by counter-clockwise rotation, given by $k_{t}.z = e^{\i t} \cdot z$, while the actions of the subgroups $A$ and $N$ have fixed point sets $\{\pm 1\}$ and $\{1\}$, respectively.

Consider the $n$-torus $\T^{n} \deq (S^{1})^{n}$ for $n \ge 1$. Its points are denoted by $\z = (z_{0},\ldots,z_{n-1})$ with $z_{j} \in S^{1}$ for all $0 \le j \le n-1$. Let us further denote by $\T^{(n)} \subset \T^{n}$ the \emph{configuration space} of configurations of $n$ pairwise distinct points on $S^{1}$, which forms the open subset of $\T^{n}$ consisting of all points $\z = (z_{0},\ldots,z_{n-1})$ satisfying $z_{i} \neq z_{j}$ for all $0 \le i < j \le n-1$. We further introduce the open subset $\To^{(n)} \subset \T^{(n)}$ of all configurations $\z = (z_{0},\ldots,z_{n-1})$ with the additional property that $z_{j} \neq 1$ for all $0 \le j \le n-1$.

The boundary action of $G$ gives rise to a smooth diagonal action of $G$ on the torus $\T^{n}$. We will always consider $\T^{n}$ as a $G$-space in this sense, denoting the action on points by $g.\z = (g.z_{0},\ldots,g.z_{n-1})$. Observe that the $G$-action on $\T^{n}$ restricts to a $G$-action on the configuration space $\T^{(n)}$. Since the boundary action of $G$ is strictly $3$-transitive, it follows that $G$ acts freely on $\T^{(n)}$ as long as $n \ge 3$. Observe moreover that, since $1$ is a fixed point for the action of $P = AN$, the $P$-action on $\T^{(n)}$ preserves the subset $\To^{(n)}$.

Let us write $\mu_{K}$ for the unique $K$-invariant probability measure on the unit circle $S^{1}$. It induces a $K$-invariant probability measure $\mu_{K}^{\otimes n}$ on the torus $\T^{n} = (S^{1})^{n}$. We fix this measure on $\T^{n}$, and will not usually indicate it in the notation. Notice that both the configuration space $\T^{(n)}$ and its open subset $\To^{(n)}$ are subspaces of full measure in $\T^{n}$.

\subsection{Coefficient modules}
\label{subsec:CoefficientModules}

Fix an integer $\m \in \Z$. We will write $\C_{\m}$ for the $K$-module $K \times \C \to \C$ defined by the standard linear action with weight $\m$ of the maximal compact subgroup $K \cong S^{1}$ on $\C$ given by
\[
k_{t}.z \deq e^{\i \m t} \cdot z
\]
for $t \in \R$ and $z \in \C$. Note that $\C_{0}$ is a trivial $K$-module, and we further define $\C_{0} = \C$ and its subspace $\R$ to be trivial $G$-modules.

\subsection{Bounded measurable functions}
\label{subsec:BoundedMeasurableFunctions}

We denote by $\Ls^{0}(\T^{n},\K)$ the space of complex measurable functions on $\T^{n}$ and by $\Ls^{\infty}(\T^{n},\K) \subset \Ls^{0}(\T^{n},\K)$ the subspace of bounded functions. Throughout this article, we will adhere to the convention from \cite{Monod/Equivariant-measurable-liftings} that $\Ls^{\infty}(\T^{n},\K)$ consists of actual bounded functions, excluding all essentially bounded functions that are not bounded.

For $p\in\{0,\infty\}$, the quotient of the space $\Ls^{p}(\T^{n},\K)$ defined by identifying functions that take the same values almost everywhere in $\T^{n}$ is denoted by $L^{p}(\T^{n},\K)$. We remind the reader that elements of this space are function classes rather than actual functions, and identities for such function classes correspond to identities for the representing functions that hold pointwise only on the complement of a subset of measure zero. Throughout we will follow the standard convention not to distinguish between functions and function classes in the notation.

We denote by $\Ls^{p}(\T^{n},\C)^{G}$ and $L^{p}(\T^{n},\C)^{G}$ the corresponding subspaces of $G$-invariant functions. Recall that functions $f$ contained in the former space satisfy $f(g.\z) = f(\z)$ for all $\z \in \T^{n}$ and $g \in G$, while functions $f$ in the latter space satisfy this identity only for almost every $\z \in \T^{n}$, for each $g \in G$.

For later reference, we observe that the canonical projection $\Ls^{\infty}(\T^{n},\K) \to L^{\infty}(\T^{n},\K)$ is $G$-equivariant and hence gives rise to a canonical map
\[
\Ls^{\infty}(\T^{n},\K)^{G} \to L^{\infty}(\T^{n},\K)^{G}.
\]
The properties of this map will be further discussed in Section \ref{subsec:EquivariantMeasurableLiftings}.

\subsection{Orbitwise smooth functions}
\label{subsec:OrbitwiseSmoothFunctions}

We recall from \cite{Hartnick/Bounded-cohomology-via-partial-differential-equations-I} the following definition.

\begin{definition}
	Let $H$ be any Lie subgroup of $G$. A measurable function $f \in \Ls^{0}(\T^{n},\K)$ is called \emph{$H$-orbitwise smooth} if for every point $\z \in \T^{n}$ the map
\begin{equation} \label{map:DefinitionSmoothness}
	H \to \K, \quad h \mapsto f(h.\z)
\end{equation}
is smooth.
\end{definition}

The space of complex $H$-orbitwise smooth measurable functions is denoted by $\Ss_{H}(\T^{n},\K)$. We will henceforth apply this concept in the cases where the subgroup $H$ is either the group $G$ itself or the parabolic subgroup $P = AN$. Notice that there is a natural inclusion $\Ss_{G}(\T^{n},\K) \subset \Ss_{P}(\T^{n},\K)$.

We denote by $\L_{K}$, $\L_{A}$ and $\L_{N}$ the fundamental vector fields for the action of the $1$-parameter subgroups $K$, $A$ and $N$ on $\T^{n}$, given pointwise by $\L_{K}(\z) = \left. \dd{}{t} \right|_{t=0} k_{t}.\z$, and likewise for $\L_{A}$ and $\L_{N}$. In order to obtain a more concrete description of these vector fields, we think of the unit circle $S^{1}$ as the quotient $S^{1} = \R/2\pi\Z$, covered by the real line $\R$. We choose a coordinate $\th \in \R$ defined by the exponential mapping $z = e^{\i \th}$ for $z\in S^{1}$. This coordinate will be called the \emph{angular coordinate} for $S^{1}$. Note that it is unique only up to multiples of $2\pi$. In this way, the torus $\T^{n}$ is endowed with angular coordinates $\thul = (\th_{0},\ldots,\th_{n-1}) \in \R^{n}$, where each coordinate $\th_{j} \in \R$ is unique up to multiples of $2\pi$. We may therefore consider functions on $\T^{n}$ as functions on $\R^{n}$ that are $2\pi$-periodic in each variable $\th_{j}$.

The boundary action of $G$ induces a smooth $G$-action on the angular coordinate $\th$, defined by the relation $g.z = e^{\i (g.\th)}$. Note in particular that $K$ acts on $\th$ by translation, and that the $K$-invariant measure $\mu_{K}$ on $S^{1}$ corresponds to the usual Lebesgue measure on $\R$, normalized by a factor of $1/2\pi$. We obtain a corresponding diagonal action of $G$ on angular coordinates for $\T^{n}$, denoted by $g.\thul = (g.\th_{0},\ldots,g.\th_{n-1})$. A short calculation (cf.~\cite[Sec.\,3.2]{Hartnick/Bounded-cohomology-via-partial-differential-equations-I}) shows that in angular coordinates the vector fields $\L_{K}$, $\L_{A}$ and $\L_{N}$ are given by
\begin{equation} \label{eqn:FundamentalVectorFields}
	\L_{K} = \sum_{j=0}^{n-1} \pd{}{\th_{j}}, \quad \L_{A} = \sum_{j=0}^{n-1} \sin(\th_{j}) \pd{}{\th_{j}}, \quad \L_{N} = \sum_{j=0}^{n-1} \bigl( 1-\cos(\th_{j}) \bigr) \pd{}{\th_{j}}
\end{equation}
and satisfy the commutator relations
\begin{equation} \label{eqn:CommutatorRelationsKAN}
	\left[ \L_{K},\L_{A} \right] = \L_{K}-\L_{N}, \quad \left[ \L_{K},\L_{N} \right] = \L_{A}, \quad \left[ \L_{A},\L_{N} \right] = \L_{N}.
\end{equation}

The fundamental vector fields give rise to first order linear partial differential operators
\[
\map{\L_{K},\L_{A},\L_{N}}{\Ss_{G}(\T^{n},\K)}{\Ss_{G}(\T^{n},\K)}
\]
and
\[
\map{\L_{A},\L_{N}}{\Ss_{P}(\T^{n},\K)}{\Ss_{P}(\T^{n},\K)}
\]
acting on orbitwise smooth functions. For example, the action of the operator $\L_{K}$ is given by $(\L_{K} f)(\z) = \left. \dd{}{t} \right|_{t=0} f(k_{t}.\z)$, and likewise for $\L_{A}$ and $\L_{N}$.

For later reference, we record the following useful formula. Consider $f \in \Ss_{P}(\T^{n},\K)$ and let $\z \in \T^{n}$. Since $t \mapsto a_{t}$ is a $1$-parameter group, we have
\[
(\L_{A} f)(a_{t}.\z) = \dd{}{t} f(a_{t}.\z).
\]
Integrating this identity then yields
\begin{equation} \label{eqn:IntegrateAlongA}
	f(a_{T}.\z) = f(\z) + \int_{0}^{T} (\L_{A} f)(a_{t}.\z) \, \d t
\end{equation}
for every $T \in \R$.

For any positive integer $\ell > 0$, and for any collection of indices $(i_{1},\ldots,i_{\ell}) \in \{K,A,N\}^{\ell}$ we define the $\ell$-th order linear partial differential operators
\[
\L_{i_{1},\ldots,i_{\ell}} \deq \L_{i_{1}} \circ \cdots \circ \L_{i_{\ell}}.
\]
Here we think of the letters $K$, $A$ and $N$ as formal indices.

\begin{definition}
	A $G$-orbitwise smooth function $f \in \Ss_{G}(\T^{n},\K)$ is said to have \emph{bounded $G$-derivatives} if all of its directional derivatives $\L_{i_{1},\ldots,i_{\ell}} f \in \Ss_{G}(\T^{n},\K)$ are bounded for all $\ell > 0$ and all $(i_{1},\ldots,i_{\ell}) \in \{K,A,N\}^{\ell}$.

Likewise, a $P$-orbitwise smooth function $f \in \Ss_{P}(\T^{n},\K)$ is said to have \emph{bounded $P$-derivatives} if its directional derivatives $\L_{i_{1},\ldots,i_{\ell}} f \in \Ss_{P}(\T^{n},\K)$ are bounded for all $\ell > 0$ and all $(i_{1},\ldots,i_{\ell}) \in \{A,N\}^{\ell}$.
\end{definition}

We denote by $\Ss^{\b}_{G}(\T^{n},\K)$ the space of all $G$-orbitwise smooth functions with bounded $G$-derivatives, and by $\Ss^{\b}_{P}(\T^{n},\K)$ the space of all $P$-orbitwise smooth functions with bounded $P$-derivatives. Notice the canonical inclusion $\Ss^{\b}_{G}(\T^{n},\K) \subset \Ss^{\b}_{P}(\T^{n},\K)$.

\subsection{$K$-equivariant functions}
\label{subsec:KEquivariantFunctions}

Recall from Section \ref{subsec:CoefficientModules} the definition of the  coefficient module $\C_{\m}$ for $\m \in \Z$. We denote by $\Ls^{0}(\T^{n},\C_{\m})^{K}$ the space of \emph{$K$-equivariant} measurable functions with values in the $K$-module $\C_{\m}$, i.e., functions $f$ satisfying
\begin{equation} \label{eqn:KEquivariance}
	f(k_{t}.\z) = e^{\i \m t} \cdot f(\z)
\end{equation}
for all $\z \in \T^{n}$ and $t\in\R$. Note that in the case $\m=0$, such functions are precisely the $K$-invariant functions. We moreover denote by $\Ss_{G}(\T^{n},\C_{\m})^{K} \subset \Ls^{0}(\T^{n},\C_{\m})^{K}$ the subspace of $G$-orbitwise smooth functions, and by $\Ss^{\b}_{G}(\T^{n},\C_{\m})^{K} \subset \Ss_{G}(\T^{n},\C_{\m})^{K}$ the subspace of functions with bounded $G$-derivatives. For later reference, we provide the following infinitesimal characterization of $K$-equivariance.

\begin{lemma} \label{lemma:InvarianceAndEquivarianceUnderK}
	Fix $\m \in \Z$. A $G$-orbitwise smooth function $f \in \Ss_{G}(\T^{n},\C_{\m})$ is $K$-equivariant if and only if it satisfies the differential equation
\begin{equation} \label{eqn:InfinitesimalKEquivariance}
	\L_{K} f - \i \m \cdot f = 0.
\end{equation}
\end{lemma}

\begin{proof}
	Let $\m \in \Z$ and $f \in \Ss_{G}(\T^{n},\C_{\m})$. Assume first that $f$ is $K$-equivariant. It follows from \eqref{eqn:KEquivariance} that
\[
(\L_{K} f) (\z) = \left. \dd{}{t} \right|_{t=0} f(k_{t}.\z) = \left. \dd{}{t} \right|_{t=0} e^{\i \m t} \cdot f(\z) = \i \m \cdot f(\z)
\]
for all $\z \in \T^{n}$ and $t \in \R$. For the converse, assume that $f$ satisfies \eqref{eqn:InfinitesimalKEquivariance}. We may then define an orbitwise smooth function $\ftilde \in \Ss_{G}(\T^{n},\C)$ by the relation
\[
f(z_{0},\ldots,z_{n-1}) = z_{0}^{\m} \cdot \ftilde(z_{0},\ldots,z_{n-1})
\]
for all $(z_{0},\ldots,z_{n-1}) \in \T^{n}$. Then
\[
0 = (\L_{K} f) (z_{0},\ldots,z_{n-1}) - \i \m \cdot f(z_{0},\ldots,z_{n-1}) = z_{0}^{\m} \cdot (\L_{K} \ftilde) (z_{0},\ldots,z_{n-1}).
\]
Hence $\L_{K} \ftilde = 0$. Since $\ftilde$ is $G$-orbitwise smooth and $G$-orbits are connected, this implies that $\ftilde$ is $K$-invariant. Thus
\[
f(k_{t}.z_{0},\ldots,k_{t}.z_{n-1}) = e^{\i \m t} \cdot z_{0}^{\m} \cdot \ftilde(z_{0},\ldots,z_{n-1}) = e^{\i \m t} \cdot f(z_{0},\ldots,z_{n-1})
\]
for all $(z_{0},\ldots,z_{n-1}) \in \T^{n}$ and $t\in\R$.
\end{proof}

To simplify notation, let us write
\[
\As(\T^{n},\C_{\m}) \deq \Ss^{\b}_{G}(\T^{n},\C_{\m})^{K}
\]
for the space of $G$-orbitwise smooth $K$-equivariant $\C_{\m}$-valued functions with bounded derivatives, and denote by $\As^{\infty}(\T^{n},\C_{\m}) \subset \As(\T^{n},\C_{\m})$ the subspace of bounded functions. We further denote by $A(\T^{n},\C_{\m})$ and $A^{\infty}(\T^{n},\C_{\m})$ the quotients of the spaces $\As(\T^{n},\C_{\m})$ and $\As^{\infty}(\T^{n},\C_{\m})$ defined by identifying functions that take the same values pointwise on the configuration space $\T^{(n)} \subset \T^{n}$. This means that representatives of a function class in $A(\T^{n},\C_{\m})$ may differ only on the complement $\T^{n} \setminus \T^{(n)}$ of the configuration space. In particular, identities for such function classes correspond to identities for the representing functions that hold pointwise on $\T^{(n)}$. Notice that this definition makes sense since the subset $\T^{(n)} \subset \T^{n}$ is invariant under the action of $G$.

We denote by $A(\T^{n},\R)^{P}$ and $A^{\infty}(\T^{n},\R)^{P}$ the spaces of $P$-invariants in the spaces $A(\T^{n},\R)$ and $A^{\infty}(\T^{n},\R)$. Function classes in these spaces are represented by functions $f$ contained in $\As(\T^{n},\R)$ and $\As^{\infty}(\T^{n},\R)$, respectively, that are $P$-invariant on the configuration space $\T^{(n)}$, i.e., they satisfy
\begin{equation} \label{eqn:GInvarianceInA}
	f(p.\z) = f(\z)
\end{equation}
for all $\z \in \T^{(n)}$ and $p \in P$. Notice that this identity is not required to hold for points $\z \in \T^{n} \setminus \T^{(n)}$ in the complement of $\T^{(n)}$. Moreover, since $G = KAN$ by the Iwasawa decomposition, it also follows that $f$ is $G$-invariant on $\T^{(n)}$, i.e., it satisfies $f(g.\z) = f(\z)$ for all $\z \in \T^{(n)}$ and $g \in G$.

\begin{definition}
	Let $\mu \in \Z$. A function $f \in \As^{\infty}(\T^{n},\C_{\m})$ is called \emph{tame} if its real part $\Re f$ satisfies
\[
\sup \left \{ \left\lvert \int_{0}^{T} (\Re f)(a_{t}.\z) \, \d t \right\rvert \,:\, T \in \R, \, \z \in \T^{n} \right\} < \infty.
\]
\end{definition}

We introduce the notation $\As^{\infty}_{\t}(\T^{n},\C_{\m}) \subset \As^{\infty}(\T^{n},\C_{\m})$ for the subspace of tame functions. The image of this space under the quotient map $\As^{\infty}(\T^{n},\C_{\m}) \to A^{\infty}(\T^{n},\C_{\m})$ will be denoted by $A^{\infty}_{\t}(\T^{n},\C_{\m})$.

\section{Cohomology}
\label{sec:Cohomology}

\subsection{Continuous bounded cohomology}
\label{subsec:ContinuousBoundedCohomology}

We briefly review some basic facts about the continuous bounded cohomology of $G$. Let us denote for all $n\ge0$ by $C_{\b}(G^{n+1},\R)$ the space of bounded continuous functions $G^{n+1} \to \R$, and let $C_{\b}(G^{n+1},\R)^{G} \subset C_{\b}(G^{n+1},\R)$ be the subspace of functions that are invariant under the diagonal action of $G$ on the product $G^{n+1}$. Then the \emph{continuous bounded cohomology} of $G$ with trivial real coefficients is defined as the cohomology of the cochain complex
\[
\begin{tikzcd}[cramped]
	0 \arrow[r,rightarrow]
	& C_{\b}(G,\R)^{G} \arrow[r,"\mathfrak{d}^{0}",rightarrow]
	& C_{\b}(G^{2},\R)^{G} \arrow[r,"\mathfrak{d}^{1}",rightarrow]
	& C_{\b}(G^{3},\R)^{G} \arrow[r,"\mathfrak{d}^{2}",rightarrow]
	& \cdots
\end{tikzcd}
\]
where the map $\map{\mathfrak{d}^{n}}{C_{\b}(G^{n+1},\R)^{G}}{C_{\b}(G^{n+2},\R)^{G}}$ given by
\[
(\mathfrak{d}^{n} f)(g_{0},\ldots,g_{n+1}) \deq \sum_{j=0}^{n+1} (-1)^{j} \cdot f(g_{0},\ldots,\widehat{g_{j}},\ldots,g_{n+1})
\]
denotes the homogeneous coboundary operator \cite{Burger/Bounded-cohomology-of-lattices-in-higher-rank-Lie-groups,Burger/Continuous-bounded-cohomology-and-applications-to-rigidity-theory,Monod/Continuous-bounded-cohomology-of-locally-compact-groups}.

The continuous bounded cohomology of $G$ is endowed with a natural ring structure determined by the \emph{cup product}
\[
\map{\smallsmile}{H_{\cb}^{n}(G;\R) \otimes H_{\cb}^{m}(G;\R)}{H_{\cb}^{n+m}(G;\R)},
\]
see for example \cite[Sec.\,1.8]{Burger/Continuous-bounded-cohomology-and-applications-to-rigidity-theory}. This cup product is induced by a corresponding cup product
\[
\map{\smallsmile}{C_{\b}(G^{n+1},\R) \otimes C_{\b}(G^{m+1},\R)}{C_{\b}(G^{n+m+1},\R)}
\]
on the level of cochains, which is defined by
\[
(c \smallsmile e)(g_{0},\ldots,g_{n+m}) \deq c(g_{0},\ldots,g_{n}) \cdot e(g_{n},g_{n+1},\ldots,g_{n+m})
\]
for any two cochains $c \in C_{\b}(G^{n+1},\R)$ and $e \in C_{\b}(G^{m+1},\R)$.

As was already mentioned in the introduction, Burger and Monod \cite[Thm.\,2.30]{BurgerIozzi/A-useful-formula-from-bounded-cohomology} proved that the comparison map is an isomorphism $H^{2}_{\cb}(G;\R) \cong H^{2}_{\c}(G;\R)$ in degree $2$. Since in our case the Lie group $G$ is Hermitian, this amounts to an isomorphism $H^{2}_{\cb}(G;\R) \cong \R$ with an explicit generator given by the \emph{bounded K\"ahler class} $\kappa \in H^{2}_{\cb}(G;\R)$. The bounded K\"ahler class is determined by a certain geometric bounded cocycle known as the \emph{Dupont cocycle} \cite[Sec.\,2.3]{BurgerIozzi/A-useful-formula-from-bounded-cohomology}.

For more background on the continuous bounded cohomology of locally compact groups we refer the reader to \cite{Monod/Continuous-bounded-cohomology-of-locally-compact-groups, Burger/Continuous-bounded-cohomology-and-applications-to-rigidity-theory, BurgerIozzi/A-useful-formula-from-bounded-cohomology}.

\subsection{The boundary model}
\label{subsec:BoundaryModel}

The approach taken in this article relies on the boundary model for the continuous bounded cohomology of $G$ due to Ivanov \cite{Ivanov/Foundations-of-the-theory-of-bounded-cohomology} and Burger and Monod \cite{Burger/Bounded-cohomology-of-lattices-in-higher-rank-Lie-groups}. Let us first consider the cochain complex
\begin{equation} \label{map:MeasurableCochainComplex}
	\begin{tikzcd}[cramped]
		0 \arrow[r,rightarrow]
		&\Ls^{0}(\T^{1},\K) \arrow[r,"\de^{0}",rightarrow]
		& \Ls^{0}(\T^{2},\K) \arrow[r,"\de^{1}",rightarrow]
		& \Ls^{0}(\T^{3},\K) \arrow[r,"\de^{2}",rightarrow]
		& \cdots
	\end{tikzcd}
\end{equation}
of complex measurable functions on the Furstenberg boundary of $G$, where
\[
\map{\de^{n}}{\Ls^{0}(\T^{n+1},\K)}{\Ls^{0}(\T^{n+2},\K)} \quad\quad (n\ge0)
\]
denotes the homogeneous coboundary operator acting by
\begin{equation} \label{eqn:CoboundaryOperator}
	(\de^{n} f)(z_{0},\ldots,z_{n+1}) \deq \sum_{j=0}^{n+1} (-1)^{j} \cdot f(z_{0},\ldots,\widehat{z_{j}},\ldots,z_{n+1}).
\end{equation}
It follows from the definitions that this coboundary operator induces coboundary operators, all denoted by the same symbol $\de^{\bullet}$, acting on the function space $L^{\infty}(\T^{\bullet+1},\R)^{G}$, as well as on the function spaces $A(\T^{\bullet+1},\R)$, $A^{\infty}(\T^{\bullet+1},\R)$, $A(\T^{\bullet+1},\R)^{P}$, $A^{\infty}(\T^{\bullet+1},\R)^{P}$, $A^{\infty}(\T^{\bullet+1},\C_{1})$ and $A_{\t}^{\infty}(\T^{\bullet+1},\C_{1})$ which we are going to work with. In this manner, we obtain corresponding cochain complexes that are subcomplexes of quotients of subcomplexes of \eqref{map:MeasurableCochainComplex}.

Of particular interest in this section is the first of these induced cochain complexes, which is the complex
\[
\begin{tikzcd}[cramped]
	0 \arrow[r,rightarrow]
	& L^{\infty}(\T^{1},\R)^{G} \arrow[r,"\de^{0}",rightarrow]
	& L^{\infty}(\T^{2},\R)^{G} \arrow[r,"\de^{1}",rightarrow]
	& L^{\infty}(\T^{3},\R)^{G} \arrow[r,"\de^{2}",rightarrow]
	& \cdots
\end{tikzcd}
\]
of $G$-invariant bounded measurable functions on the Furstenberg boundary of $G$. The boundary model realizes the continuous bounded cohomology of $G$ in terms of this cochain complex. More specifically, since the boundary action of $G$ is amenable, by \cite[Thm.\,7.5.3]{Monod/Continuous-bounded-cohomology-of-locally-compact-groups} there is an isomorphism
\begin{equation} \label{map:BoundaryModel}
	H^{n}_{\cb}(G;\R) \cong H^{n}\left( L^{\infty}(\T^{\bullet+1},\R)^{G},\de^{\bullet} \right)
\end{equation}
in every degree $n \ge 0$. Notice that this collection of isomorphisms is compatible with the cup product introduced in Section \ref{subsec:ContinuousBoundedCohomology} and hence gives rise to an isomorphism of the respective bounded cohomology rings.

\subsection{Equivariant measurable liftings}
\label{subsec:EquivariantMeasurableLiftings}

First of all, we observe that the canonical map
\begin{equation} \label{map:EquivariantLifting}
	\Ls^{\infty}(\T^{n+1},\R)^{G} \to L^{\infty}(\T^{n+1},\R)^{G} \quad\quad (n \ge 0)
\end{equation}
considered in Section \ref{subsec:BoundedMeasurableFunctions} is in fact a cochain map that intertwines with the action of the coboundary operator $\de$. As Monod \cite{Monod/Equivariant-measurable-liftings} points out, there is a priori no reason for this map to be surjective. However, since the boundary action of $G$ is amenable, by a result of Monod \cite[Thm.\,A, Rem.\,1 and Cor.\,6]{Monod/Equivariant-measurable-liftings} the cochain map in \eqref{map:EquivariantLifting} does in fact admit a section that intertwines with $\de$. An immediate consequence of this is the following proposition, which paves the way for applying differential geometric methods in the study of the boundary model for the continuous bounded cohomology of $G$.

\begin{proposition}
\label{prop:EquivariantLifting}
	There is a surjective homomorphism
\begin{equation} \label{map:EquivariantLiftingOnCohomology}
	H^{n}\left( A^{\infty}(\T^{\bullet+1},\R)^{P},\de^{\bullet} \right) \to H^{n}\left( L^{\infty}(\T^{\bullet+1},\R)^{G},\de^{\bullet} \right)
\end{equation}
for every $n \ge 0$.
\end{proposition}

\begin{proof}
	Recall from Section \ref{subsec:KEquivariantFunctions} that every function $f \in A^{\infty}(\T^{n+1},\R)^{P}$ has a representative $f \in \As^{\infty}(\T^{n+1},\R)$ that is bounded and $G$-invariant on the configuration space $\T^{(n+1)}$. Since the subspace $\T^{(n+1)}$ is of full measure in $\T^{n+1}$, this function $f$ therefore defines an element of $L^{\infty}(\T^{n+1},\R)^{G}$. Hence we obtain a natural cochain map
\begin{equation} \label{map:EquivariantLiftingFromA^P}
	A^{\infty}(\T^{n+1},\R)^{P} \to L^{\infty}(\T^{n+1},\R)^{G} \quad\quad (n \ge 0)
\end{equation}
that intertwines with $\de$. This map admits a section that is induced by the section of the cochain map in \eqref{map:EquivariantLifting} \cite[Thm.\,A, Rem.\,1 and Cor.\,6]{Monod/Equivariant-measurable-liftings}. To see this, we note that every $G$-invariant bounded measurable function $f \in \Ls^{\infty}(\T^{n+1},\R)^{G}$ is constant along all $G$-orbits in $\T^{n+1}$, hence it is in particular $G$-orbitwise smooth and bounded with bounded $G$-derivatives. Thus it determines a function in the space $A^{\infty}(\T^{n+1},\R)^{P}$.
\end{proof}

\subsection{Cochain contractions}
\label{subsec:CochainContractions}

For every $\m \in \Z$, we define a linear integral operator
\[
\map{\I^{n}}{A^{\infty}(\T^{n+1},\C_{\m})}{A^{\infty}(\T^{n},\C_{\m})} \quad\quad (n \ge 1)
\]
by
\begin{equation} \label{eqn:DefinitionOfOperatorI}
	(\I^{n} f)(z_{0},\ldots,z_{n-1}) \deq \int_{S^{1}} f(z,z_{0},\ldots,z_{n-1}) \, \d\mu_{K}(z).
\end{equation}
We are now going to prove that this operator is well-defined and gives rise to a cochain contraction for the complex $(A^{\infty}(\T^{\bullet+1},\C_{\m}),\de^{\bullet})$. Recall that this means that the operator $\I$ satisfies the identity
\begin{equation} \label{eqn:CochainContraction}
	\I^{n+1} \circ \de^{n} + \de^{n-1} \circ \I^{n} = \Id
\end{equation}
for all $n>0$.

\begin{proposition} \label{prop:ComplexAIsAcyclic}
	For every $\m \in \Z$, the operator $\I$ is a well-defined cochain contraction for the complex $(A^{\infty}(\T^{\bullet+1},\C_{\m}),\de^{\bullet})$.
\end{proposition}

\begin{proof}
	Fix $\m \in \Z$. First of all, we check that the operator $\I$ is well-defined. Consider a function $f \in \As^{\infty}(\T^{n+1},\C_{\m})$ representing a class in $A^{\infty}(\T^{n+1},\C_{\m})$.

Since $f$ is bounded, the integral in \eqref{eqn:DefinitionOfOperatorI} exists for every point $(z_{0},\ldots,z_{n-1}) \in \T^{n}$ and defines a bounded measurable function $\I f$ on $\T^{n}$. Then $K$-invariance of the measure $\mu_{K}$ together with $K$-equivariance of $f$ imply that the function $\I f$ is $K$-equivariant.

Next we prove that $\I f$ is $G$-orbitwise smooth and has bounded derivatives. Let us begin by considering the first order derivatives of $\I f$. Since $\I f$ is $K$-equivariant, it follows from Lemma~\ref{lemma:InvarianceAndEquivarianceUnderK} that the derivative $\L_{K} \I f$ exists and is bounded on $\T^{n}$. We now inspect the derivatives $\L_{A} \I f$ and $\L_{N} \I f$. This requires some computations, which are best carried out in angular coordinates $(\th_{0},\ldots,\th_{n-1}) \in \T^{n}$. Formally, we have
\begin{equation} \label{eqn:ADerivativeOfIf1}
\begin{aligned}
	(\L_{A} \I f)(\th_{0},\ldots,\th_{n-1}) &= \left. \dd{}{t}\right|_{t=0} (\I f)(a_{t}.\th_{0},\ldots,a_{t}.\th_{n-1}) \\
	&= \frac{1}{2\pi} \left. \dd{}{t}\right|_{t=0} \, \int_{0}^{2\pi} f(\eta,a_{t}.\th_{0},\ldots,a_{t}.\th_{n-1}) \, \d\eta \\
	&= \frac{1}{2\pi} \int_{0}^{2\pi} \, \left. \dd{}{t}\right|_{t=0} \left( \dd{(a_{t}.\eta)}{\eta} \cdot f(a_{t}.\eta,a_{t}.\th_{0},\ldots,a_{t}.\th_{n-1}) \right) \d\eta.
\end{aligned}
\end{equation}
Since $f$ is $G$-orbitwise smooth, the derivative appearing under the integral sign is given by
\begin{equation} \label{eqn:ADerivativeOfIf2}
\begin{aligned}
	& \left. \dd{}{t}\right|_{t=0} \left( \dd{(a_{t}.\eta)}{\eta} \cdot f(a_{t}.\eta,a_{t}.\th_{0},\ldots,a_{t}.\th_{n-1}) \right) \\
	&= \left. \dd{}{t}\right|_{t=0} f(a_{t}.\eta,a_{t}.\th_{0},\ldots,a_{t}.\th_{n-1}) + \left. \dd{}{t}\right|_{t=0} \dd{(a_{t}.\eta)}{\eta} \cdot f(\eta,\th_{0},\ldots,\th_{n-1}) \\
	&= (\L_{A} f)(\eta,\th_{0},\ldots,\th_{n-1}) + \cos(\eta) \cdot f(\eta,\th_{0},\ldots,\th_{n-1}),
\end{aligned}
\end{equation}
where in the last step we used the identity
\[
\left. \dd{}{t}\right|_{t=0} \dd{(a_{t}.\eta)}{\eta} = \dd{}{\eta} \left. \dd{(a_{t}.\eta)}{t}\right|_{t=0} = \dd{}{\eta} \sin(\eta) = \cos(\eta)
\]
which follows from \eqref{eqn:FundamentalVectorFields}. Since the function $f$ is bounded with bounded derivatives, we see that the derivative in \eqref{eqn:ADerivativeOfIf2} is bounded on $\T^{n}$. Hence the Lebesgue dominated convergence theorem justifies the computation in \eqref{eqn:ADerivativeOfIf1} and therefore the derivative $\L_{A} \I f$ exists. Further, combining \eqref{eqn:ADerivativeOfIf1} and \eqref{eqn:ADerivativeOfIf2} we obtain the formula
\begin{multline} \label{eqn:FormulaL_AI}
	(\L_{A} \I f)(\th_{0},\ldots,\th_{n-1}) = (\I \L_{A} f)(\th_{0},\ldots,\th_{n-1}) \\ + \frac{1}{2\pi} \, \int_{0}^{2\pi} \cos(\eta) \cdot f(\eta,\th_{0},\ldots,\th_{n-1}) \, \d\eta.
\end{multline}
Likewise we have
\begin{multline} \label{eqn:FormulaL_NI}
	(\L_{N} \I f)(\th_{0},\ldots,\th_{n-1}) = (\I \L_{N} f)(\th_{0},\ldots,\th_{n-1}) \\ + \frac{1}{2\pi} \, \int_{0}^{2\pi} \sin(\eta) \cdot f(\eta,\th_{0},\ldots,\th_{n-1}) \, \d\eta.
\end{multline}
Since $f$ is bounded with bounded derivatives it follows that the derivatives $\L_{A} \I f$ and $\L_{N} \I f$ are bounded functions on $\T^{n}$. In the general case, a similar argument shows that the directional derivatives $\L_{i_{1},\ldots,i_{\ell}} \I f$ are bounded functions on $\T^{n}$ for all $\ell > 0$ and all $(i_{1},\ldots,i_{\ell}) \in \{K,A,N\}^{\ell}$. Since $G = KAN$ by the Iwasawa decomposition, this proves that $\I f$ is $G$-orbitwise smooth and bounded with bounded derivatives.

We see from \eqref{eqn:DefinitionOfOperatorI} that the restriction of $\I f$ to the configuration space $\T^{(n)}$ does not depend on the choice of $f$ since $f$ is uniquely determined on the configuration space $\T^{(n+1)}$ and $(z,z_{0},\ldots,z_{n-1}) \in \T^{(n+1)}$ for almost every $z \in S^{1}$. Hence the function $\I f$ defines a class in $A^{\infty}(\T^{n},\C_{\m})$.

A straightforward calculation shows that $\I^{n+1} \circ \de^{n} f + \de^{n-1} \circ \I^{n} f = f$ holds pointwise on the configuration space $\T^{(n+1)}$ for every function $f \in \As(\T^{n+1},\C_{\m})$ and for all $n \ge 0$.
\end{proof}

An immediate consequence of the proposition is the following vanishing theorem for the cohomology of the complex $(A^{\infty}(\T^{\bullet+1},\C_{\m}),\de^{\bullet})$.

\begin{corollary} \label{cor:VanishingForA}
	For every $\m \in \Z$, the cochain complex $(A^{\infty}(\T^{\bullet+1},\C_{\m}),\de^{\bullet})$ is acyclic and hence $H^{n}(A^{\infty}(\T^{\bullet+1},\C_{\m}),\de^{\bullet}) = 0$ for all $n>0$.
\end{corollary}

\section{The Cauchy-Frobenius complex}
\label{sec:CauchyFrobeniusComplex}

\subsection{The differential operators $\L$ and $\Q$}
\label{subsec:OperatorsLAndQ}

We introduce two basic first order linear partial differential operators acting on $P$-orbitwise smooth functions. The first operator is defined by combining the real operators $\L_{A}$ and $\L_{N}$, introduced in Section \ref{subsec:OrbitwiseSmoothFunctions}, into a single complex operator.

\begin{definition} \label{def:OperatorL}
	The \emph{Cauchy operator} is the complex operator
\begin{equation} \label{map:OperatorLUpstairs}
	\map{\L \deq \L_{A} + \,\i \L_{N}}{\Ss_{P}(\T^{n+1},\R)}{\Ss_{P}(\T^{n+1},\C)}.
\end{equation}
\end{definition}

The Cauchy operator naturally acts on $P$-orbitwise smooth functions. Its complex conjugate will be denoted by $\Lbar \deq \L_{A} - \,\i \L_{N}$. For later reference, we note that as an immediate consequence of the real commutator relations in \eqref{eqn:CommutatorRelationsKAN}, the operators $\L_{K}$, $\L$ and $\Lbar$ satisfy the complex commutator relations
\begin{equation} \label{eqn:CommutatorRelationsL}
	\left[ \L_{K},\L \right] - \L_{K} - \,\i\L = 0, \quad \left[ \L_{K},\Lbar \right] - \L_{K} + \,\i\Lbar = 0, \quad \left[ \L,\Lbar \right] + \L - \,\Lbar = 0.
\end{equation}
The second operator is defined in terms of the conjugated Cauchy operator $\Lbar$.

\begin{definition}  \label{def:OperatorQ}
	The \emph{Frobenius operator} is the real operator
\begin{equation} \label{map:OperatorQUpstairs}
	\map{\Q \deq \Im\bigl( \Id - \Lbar \bigr)}{\Ss_{P}(\T^{n+1},\C)}{\Ss_{P}(\T^{n+1},\R)}
\end{equation}
defined as the imaginary part of the operator $\Id - \Lbar$.
\end{definition}

We reserve the notation $u = \usharp + \i \uflat$ for the decomposition of a complex function $u$ into its real and imaginary parts. For later reference, we note that the action of the Frobenius operator on some function $u \in \Ss_{P}(\T^{n+1},\C)$ then takes the form
\begin{equation} \label{eqn:RealVersionActionOfQ}
	\Q u = \uflat - \L_{A} \uflat + \L_{N} \usharp.
\end{equation}

\subsection{The Cauchy-Frobenius complex}

The goal of this section is to investigate the interaction between the differential operators $\L$ and $\Q$. We denote by
\[
\map{\iota^{n}}{A^{\infty}(\T^{n+1},\R)^{P}}{A^{\infty}(\T^{n+1},\R)} \quad\quad (n \ge 0)
\]
the canonical inclusion. We begin with the following basic observation.

\begin{proposition} \label{prop:CauchyFrobeniusSequence}
	The differential operators $\L$ and $\Q$ in \eqref{map:OperatorLUpstairs} and \eqref{map:OperatorQUpstairs} induce linear operators
\begin{equation} \label{map:OperatorLDownstairs}
	\map{\L^{n}}{A(\T^{n+1},\R)}{A^{\infty}(\T^{n+1},\C_{1})} \quad\quad (n \ge 0)
\end{equation}
and
\begin{equation} \label{map:OperatorQDownstairs}
	\map{\Q^{n}}{A^{\infty}(\T^{n+1},\C_{1})}{A^{\infty}(\T^{n+1},\R)} \quad\quad (n \ge 0)
\end{equation}
which give rise to a differential complex
\begin{equation} \label{map:CauchyFrobeniusComplex}
	\begin{tikzcd}[column sep = small]
		0 \arrow[r,rightarrow] & A^{\infty}(\T^{n+1},\R)^{P} \arrow[r,"\iota^{n}",rightarrow]
		& A^{\infty}(\T^{n+1},\R) \arrow[r,"\L^{n}",rightarrow]
		& A^{\infty}_{\t}(\T^{n+1},\C_{1}) \arrow[r,"\Q^{n}",rightarrow]
		& A^{\infty}(\T^{n+1},\R) \arrow[r,rightarrow] & 0
	\end{tikzcd}
\end{equation}
for every $n \ge 0$.
\end{proposition}

The differential complex in \eqref{map:CauchyFrobeniusComplex} will be called the \emph{Cauchy-Frobenius complex}.

\begin{proof}
	\noindent{\textbf{Step 1.}} We prove that the Cauchy operator in \eqref{map:OperatorLUpstairs} induces a linear operator
\[
\map{\L^{n}}{A(\T^{n+1},\R)}{A^{\infty}(\T^{n+1},\C_{1})}.
\]

\medskip

Consider a function $p \in A(\T^{n+1},\R)$ that is represented by some function $p \in \As(\T^{n+1},\R)$. The function $\L p$ is bounded with bounded derivatives since $p$ has bounded derivatives. By Lemma \ref{lemma:InvarianceAndEquivarianceUnderK} we have $\L_{K} p = 0$. Hence it follows with the commutator relations from \eqref{eqn:CommutatorRelationsL} that
\[
\L_{K} (\L p) - \i \, \L p = \left[ \L_{K},\L \right] p - \L_{K} p - \,\i\L p = 0,
\]
which by Lemma \ref{lemma:InvarianceAndEquivarianceUnderK} implies that $\L p$ is $K$-equivariant as a function taking values in the $K$-module $\C_{1}$. Thus the function $\L p$ determines a well-defined class in $A^{\infty}(\T^{n+1},\C_{1})$ since the configuration space $\T^{(n+1)}$ is invariant under the action of $P$.

\medskip

\noindent{\textbf{Step 2.}} We prove that the operator $\L^{n}$ from Step 1 restricts to an operator
\[
\map{\L^{n}}{A^{\infty}(\T^{n+1},\R)}{A^{\infty}_{\t}(\T^{n+1},\C_{1})}.
\]

\medskip

Continuing with the argument from Step 1, it remains to check that $\L p$ is tame. To this end, we observe that $\Re(\L p) = \L_{A} p$. Hence by \eqref{eqn:IntegrateAlongA} we obtain for $\z \in \T^{(n+1)}$ the identity
\[
\int_{0}^{T} (\Re(\L p))(a_{t}.\z) \, \d t = p(a_{T}.\z) - p(\z)
\]
for all $T \in \R$. Since $p$ is bounded we conclude that $\L p$ is tame.

\medskip

\noindent{\textbf{Step 3.}} We prove that $\im \, \iota^{n} \subset \ker \L^{n}$.

\medskip

Consider a function $p \in A^{\infty}(\T^{n+1},\R)^{P}$. It is $G$-orbitwise smooth and $P$-invariant when restricted to the configuration space $\T^{(n+1)}$, hence invariant under the actions of $A$ and $N$ thereon. Thus $\L_{A} p = 0 = \L_{N} p$ on $\T^{(n+1)}$, which implies that $\L \iota p = 0$ in $A(\T^{n+1},\C_{1})$ because $p$ is real valued.

\medskip

\noindent{\textbf{Step 4.}} We prove that the Frobenius operator in \eqref{map:OperatorQUpstairs} induces a linear operator
\[
\map{\Q^{n}}{A^{\infty}(\T^{n+1},\C_{1})}{A^{\infty}(\T^{n+1},\R)}.
\]

\medskip

Consider a function $u \in A^{\infty}(\T^{n+1},\C_{1})$ represented by some function $u \in \As^{\infty}(\T^{n+1},\C_{1})$. The function $\Q u$ is bounded with bounded derivatives for $u$ has bounded derivatives. Since $u$ is $K$-equivariant with values in $\C_{1}$, by Lemma \ref{lemma:InvarianceAndEquivarianceUnderK} we have $\L_{K} u - \i\, u = 0$. Thus it follows with \eqref{eqn:CommutatorRelationsL} that
\[
\L_{K} \Q u = \L_{K} ( u - \,\Lbar u ) = \L_{K} u - \left[ \L_{K},\Lbar \right] u - \Lbar \L_{K} u = \i\, u - \L_{K} u + \i\,\Lbar u - \i\,\Lbar u = 0.
\]
By Lemma \ref{lemma:InvarianceAndEquivarianceUnderK} this implies that $\Q u$ is $K$-equivariant as a $\C_{0}$-valued function, hence $K$-invariant. As in Step~1 we see that $\Q u$ defines a class in $A^{\infty}(\T^{n+1},\R)$.

\medskip

\noindent{\textbf{Step 5.}} We prove that $\im \L^{n} \subset \ker \Q^{n}$.

\medskip

Let $p \in A^{\infty}(\T^{n+1},\R)$. Using the commutator relations from \eqref{eqn:CommutatorRelationsL} we compute
\begin{equation*} \label{eqn:PLq}
	\Q \L p = \Im \left( \L p - \Lbar \L p \right) = \frac{1}{2\,\i} \left( \left[ \L,\Lbar \right] p + \L p - \,\Lbar p \right) = 0. \qedhere
\end{equation*}
\end{proof}

The next proposition, which is the main result of this section, characterizes the interaction between the differential operators $\L$ and $\Q$. Its proof will occupy the remainder of this section.

\begin{proposition} \label{prop:ExactnessOfCauchyFrobeniusComplex}
	The Cauchy-Frobenius complex in \eqref{map:CauchyFrobeniusComplex} is exact for every $n \ge 2$. Moreover, for $n=1$ it is exact at the last term, i.e., the map $\Q^{1}$ is surjective.
\end{proposition}

\begin{proof}
	Exactness of the Cauchy-Frobenius complex at the first term is clear. Exactness at the other terms holds by Proposition \ref{prop:InfinitesimalPInvariance}, Proposition \ref{prop:CauchyProblem} and Proposition \ref{prop:FrobeniusProblem} below.
\end{proof}

\subsection{Infinitesimal $P$-invariance}
\label{subsec:InfinitesimalPInvariance}

\begin{proposition} \label{prop:InfinitesimalPInvariance}
	In \eqref{map:CauchyFrobeniusComplex} we have $\im \, \iota^{n} = \ker \L^{n}$ for every $n \ge 1$.
\end{proposition}

\begin{proof}
	By Proposition \ref{prop:CauchyFrobeniusSequence}, it remains to show that $\ker \L^{n} \subset \im \, \iota^{n}$. So consider a function $p \in A^{\infty}(\T^{n+1},\R)$. Since $p$ is real valued, $\L^{n} p = 0$ implies that $\L_{A} p = 0 = \L_{N} p$ on the configuration space $\T^{(n+1)}$. Since $p$ is smooth along $P$-orbits and $P$-orbits are connected, it follows that $p$ is invariant under the actions of $A$ and $N$ on $\T^{(n+1)}$, hence $P$-invariant thereon. We conclude that $p \in A^{\infty}(\T^{n+1},\R)^{P}$.
\end{proof}

\subsection{$K$-reduction and $K$-extension}
\label{subsec:KReductionAndExtension}

We introduce the concepts of $K$-reduction and $K$-extension, which will be useful when dealing with differential equations for $K$-equivariant functions. Let $n \ge 0$. Given a measurable function $f \in \Ls^{0}(\T^{n+1},\C)$, the \emph{$K$-reduction} of $f$ is the function $f_{K} \in \Ls^{0}(\T^{n},\C)$ defined by
\begin{equation} \label{eqn:KReduction}
	f_{K}(z_{1},\ldots,z_{n}) \deq f(1,z_{1},\ldots,z_{n}).
\end{equation}
Conversely, given a weight $\m \in \Z$ and a function $f \in \Ls^{0}(\T^{n},\K)$, the \emph{$K$-extension of $f$ with weight $\m$} is the function $f^{K}_{\m} \in \Ls^{0}(\T^{n+1},\C_{\m})^{K}$ defined by
\begin{equation} \label{eqn:KExtension}
	f^{K}_{\m}(z_{0},\ldots,z_{n}) \deq z_{0}^{\m} \cdot f(z_{1}/z_{0},\ldots,z_{n}/z_{0}).
\end{equation}
The next lemma collects some basic properties of $K$-reduction and $K$-extension.

\begin{lemma} \label{lemma:ReductionAndExtensionByK}
	Let $n \ge 0$, and fix an integer $\m \in \Z$.
\begin{enumerate}[leftmargin=1cm,topsep=0.5ex,itemsep=0.5ex]
	\item
For all $f \in \Ls^{0}(\T^{n+1},\C_{\m})^{K}$ we have $(f_{K})^{K}_{\m} = f$, and for all $f \in \Ls^{0}(\T^{n},\C)$ we have $(f^{K}_{\m})_{K} = f$.
	\item Let $f, f^{\prime} \in \Ls^{0}(\T^{n+1},\C_{\m})^{K}$. If $f_{K} = f^{\prime}_{K}$, then $f=f^{\prime}$.
	\item Let $f \in \Ls^{0}(\T^{n+1},\C)$. Then $f$ is bounded if and only if $f_{K}$ is bounded if and only if $f^{K}_{\m}$ is bounded.
	\item If $f \in \Ss_{G}(\T^{n+1},\C)$, then $f_{K} \in \Ss_{P}(\T^{n},\C)$. Moreover, we have
\[
\L_{A} f_{K} = (\L_{A} f)_{K}, \quad \L_{N} f_{K} = (\L_{N} f)_{K}.
\]
	\item If $f \in \Ss^{\b}_{G}(\T^{n+1},\C)$, then $f_{K} \in \Ss^{\b}_{P}(\T^{n},\C)$.
	\item If $f \in \Ss_{P}(\T^{n},\C)$, then $f^{K}_{\m} \in \Ss_{G}(\T^{n+1},\C_{\m})^{K}$.
	\item If $f \in \Ss^{\b}_{P}(\T^{n},\C)$, then $f^{K}_{\m} \in \Ss^{\b}_{G}(\T^{n+1},\C_{\m})^{K}$.
\end{enumerate}
\end{lemma}

\begin{proof}
	Claims (i) and (iii) are immediate from \eqref{eqn:KReduction} and \eqref{eqn:KExtension}, and (ii) follows from (i).

\medskip

To prove (iv), recall that for a $G$-orbitwise smooth function $f \in \Ss_{G}(\T^{n+1},\C)$ the map
\[
G \to \K, \quad g \mapsto f(g.z_{0},\ldots,g.z_{n})
\]
is smooth for every $(z_{0},\ldots,z_{n}) \in \T^{n+1}$. Recall moreover that $1 \in S^{1}$ is a fixed point for the action of the parabolic subgroup $P=AN$. Firstly, this implies that the map
\[
P \to \K, \quad p \mapsto f(p.1,p.z_{1},\ldots,p.z_{n}) = f_{K}(p.z_{1},\ldots,p.z_{n})
\]
is smooth, which shows that $f_{K} \in \Ss_{P}(\T^{n},\K)$. Secondly, it implies that $K$-reduction commutes with the action of the operators $\L_{A}$ and $\L_{N}$. In fact, for $(z_{1},\ldots,z_{n}) \in \T^{n}$ we have
\[
\begin{aligned}
	(\L_{A} f_{K})(z_{1},\ldots,z_{n}) &= \left. \dd{}{t} \right|_{t=0} f_{K}(a_{t}.z_{1},\ldots,a_{t}.z_{n}) \\
	&= \left. \dd{}{t} \right|_{t=0} f(a_{t}.1,a_{t}.z_{1},\ldots,a_{t}.z_{n}) \\
	&= (\L_{A} f)(1,z_{1},\ldots,z_{n}) \,=\, (\L_{A} f)_{K}(z_{1},\ldots,z_{n}),
\end{aligned}
\]
and likewise for $\L_{N}$.

\medskip

Let us prove (v). Assume that $f \in \Ss^{\b}_{G}(\T^{n+1},\C)$. By (iv) above it remains to show that $f_{K}$ has bounded $P$-derivatives. By (iv) we have
\[
\L_{i_{1},\ldots,i_{\ell}} f_{K} = ( \L_{i_{1},\ldots,i_{\ell}} f )_{K}
\]
for all $\ell > 0$ and all $(i_{1},\ldots,i_{\ell}) \in \{A,N\}^{\ell}$. The claim now follows with (iii) above since $f$ has bounded $G$-derivatives.

\medskip

To prove claim (vi) we consider $f \in \Ss_{P}(\T^{n},\K)$ and let $(z_{0},\ldots,z_{n}) \in \T^{n+1}$. We are going to show that the map
\begin{equation} \label{map:SmoothnessOfKExtension}
	G \to \K, \quad g \mapsto f^{K}_{\m}(g.z_{0},\ldots,g.z_{n})
\end{equation}
is smooth. Let us fix $t \in \R$ such that $k_{t} = z_{0} \in S^{1} \cong K$. Then $k_{t}^{-1}.z_{0}=1$. From the Iwasawa decomposition $G = KAN = KP$ we obtain the decomposition $G = K P^{\prime}$ with the parabolic subgroup $P^{\prime} \deq k_{t}\,P\,k_{t}^{-1}$. Any $g \in G$ may therefore be written in the form
\begin{equation} \label{eqn:ConjugateIwasawa}
	g = k\,k_{t}\,p\,k_{t}^{-1},
\end{equation}
with $k \in K$ and $p \in P$ smoothly depending on $g$. Let us write $k \, k_{t} = e^{t^{\prime}i}$ with $t^{\prime} \in \R$ smoothly depending on $g$. Then it follows from \eqref{eqn:KExtension} and the fact that $1 \in S^{1}$ is a fixed point for the action of $P$ that
\[
\begin{aligned}
	f^{K}_{\m}(g.z_{0},\ldots,g.z_{n}) &= e^{\i \m t^{\prime}} \cdot f^{K}_{\m}\left( p\,k_{t}^{-1}.z_{0},p\,k_{t}^{-1}.z_{1},\ldots,p\,k_{t}^{-1}.z_{n} \right) \\
	&= e^{\i \m t^{\prime}} \cdot f^{K}_{\m}\left( 1,p\,k_{t}^{-1}.z_{1},\ldots,p\,k_{t}^{-1}.z_{n} \right) \\
	&= e^{\i \m t^{\prime}} \cdot f\left( p\,k_{t}^{-1}.z_{1},\ldots,p\,k_{t}^{-1}.z_{n} \right).
\end{aligned}
\]
Since the function $f$ is $P$-orbitwise smooth and $t^{\prime}$ and $p$ depend smoothly on $g$, we conclude that the map \eqref{map:SmoothnessOfKExtension} is in fact smooth.

\medskip

Lastly, we prove (vii). Let $f \in \Ss^{\b}_{P}(\T^{n},\K)$. By (vi) above it remains to show that the function $f^{K}_{\m}$ has bounded $G$-derivatives. To this end, let us first introduce some notation. We abbreviate $\E_{0} \deq \L_{K}$, $\E_{1} \deq \L$ and $\E_{2} \deq \Lbar$. Given an integer $\ell > 0$, for any collection of indices $(j_{1},\ldots,j_{\ell}) \in \{0,1,2\}^{\ell}$ we then consider the $\ell$-th order linear partial differential operators
\begin{equation*}
	\E_{j_{1},\ldots,j_{\ell}} \deq \E_{j_{1}} \circ \cdots \circ \E_{j_{\ell}}.
\end{equation*}
For $\ell = 0$ we set $\E_{j_{1},\ldots,j_{\ell}} \deq \Id$. Observe that any of the differential operators $\L_{i_{1},\ldots,i_{\ell}}$ defined in Section \ref{subsec:OrbitwiseSmoothFunctions} may be expressed as a complex linear combination of the differential operators $\E_{j_{1},\ldots,j_{\ell}}$. Hence in order to prove that $f^{K}_{\m}$ has bounded derivatives it will be sufficient to show that the derivatives $\E_{j_{1},\ldots,j_{\ell}} f^{K}_{\m}$ are bounded for all $\ell > 0$ and all $(j_{1},\ldots,j_{\ell}) \in \{0,1,2\}^{\ell}$.

Let us consider the first order derivatives of the function $f^{K}_{\m}$. Since by (v) above $f^{K}_{\m}$ is $K$-equivariant, by Lemma \ref{lemma:InvarianceAndEquivarianceUnderK} we have
\[
\E_{0} f^{K}_{\m} = \L_{K} f^{K}_{\m} = \i \m \cdot f^{K}_{\m},
\]
which is bounded since $f^{K}_{\m}$ is bounded. Let now $j_{1} \in \{1,2\}$. Using the commutator relations from \eqref{eqn:CommutatorRelationsL} we arrive at the differential equation
\begin{equation} \label{eqn:EqnBoundednessOfDerivativesFirstOrder}
	\L_{K} \left( \E_{j_{1}} f^{K}_{\m} \right) = \i \nu \cdot \left( \E_{j_{1}} f^{K}_{\m} \right) + R_{0}(f^{K}_{\m})
\end{equation}
for the derivative $\E_{j_{1}} f^{K}_{\m}$, with $\nu \in \Z$ and the lower order perturbation term
\[
R_{0}(f^{K}_{\m}) = \i \beta \cdot f^{K}_{\m}
\]
for some $\beta \in \Z$. Notice that \eqref{eqn:EqnBoundednessOfDerivativesFirstOrder} is a first order linear ordinary differential equation along each $K$-orbit in $\T^{n+1}$. By Lemma \ref{lemma:InvarianceAndEquivarianceUnderK} and (iii) above, any solution of the unperturbed equation in \eqref{eqn:EqnBoundednessOfDerivativesFirstOrder} is bounded if and only if its $K$-reduction is bounded. Observe moreover that the perturbation term in \eqref{eqn:EqnBoundednessOfDerivativesFirstOrder} is bounded. Since $K \cong S^{1}$ is compact, we therefore conclude that the solution $\E_{j_{1}} f^{K}_{\m}$ of the perturbed equation in \eqref{eqn:EqnBoundednessOfDerivativesFirstOrder} is bounded if and only if its $K$-reduction is bounded (cf.\,\cite[Sec.\,3.3]{Arnold/Ordinary-differential-equations}). Now by (iv) and (i) above this $K$-reduction is given by
\[
(\E_{j_{1}} f^{K}_{\m})_{K} = \E_{j_{1}} (f^{K}_{\m})_{K} = \E_{j_{1}} f,
\]
which is bounded since $f$ has bounded $P$-derivatives. Hence the derivatives $\E_{j_{1}} f^{K}_{\m}$ are bounded for $j_{1} \in \{0,1,2\}$.

We may now consider derivatives of the function $f^{K}_{\m}$ of any order $\ell > 1$. To this end, we let $(j_{1},\ldots,j_{\ell}) \in \{0,1,2\}^{\ell}$ and inductively apply the commutator relations from \eqref{eqn:CommutatorRelationsL} to obtain the differential equation
\begin{equation*} \label{eqn:EqnBoundednessOfDerivatives}
	\L_{K} \left( \E_{j_{1},\ldots,j_{\ell}} f^{K}_{\m} \right) = \i \gamma \cdot \E_{j_{1}\ldots,j_{\ell}} f^{K}_{\m} + R_{\ell-1}(f^{K}_{\m})
\end{equation*}
for the derivative $\E_{j_{1},\ldots,j_{\ell}} f^{K}_{\m}$, with $\gamma \in \Z$ and the lower order perturbation term
\[
R_{\ell-1}(f^{K}_{\m}) = \sum_{0 \le \kappa < \ell} \,\, \sum_{(l_{1},\ldots,l_{\kappa}) \in \{0,1,2\}^{\kappa}} \i \alpha_{l_{1},\ldots,l_{\kappa}} \cdot \E_{l_{1},\ldots,l_{\kappa}} f^{K}_{\m}
\]
with $\alpha_{l_{1},\ldots,l_{\kappa}} \in \Z$. It follows by induction that the function $R_{\ell-1}(f^{K}_{\m})$ is bounded. Hence a similar argument as in the case $\ell=1$ above shows that the derivative $\E_{j_{1},\ldots,j_{\ell}} f^{K}_{\m}$ is in fact bounded.
\end{proof}

We will also need the following useful criterion for tameness.

\begin{lemma} \label{lemma:CriterionForTameness}
	If the real part of a bounded function $f \in \Ss^{\b}_{P}(\T^{n},\C)$ satisfies $\Re f = 0$, then the $K$-extension $f^{K}_{1} \in \Ss^{\b}_{G}(\T^{n+1},\C_{1})^{K}$ of $f$ with weight $1$ is tame.
\end{lemma}

\begin{proof}
	Let $f \in \Ss^{\b}_{P}(\T^{n},\C)$, and assume that $\Re f = 0$. By Lemma \ref{lemma:ReductionAndExtensionByK}\,(iii) we know that $f^{K}_{1}$ is bounded since $f$ is bounded by assumption. It will be convenient to work with angular coordinates $(\th_{0},\ldots,\th_{n}) \in \T^{n+1}$. Recall from \eqref{eqn:KExtension} that the $K$-extension $f^{K}_{1} \in \Ss^{\b}_{G}(\T^{n+1},\C_{1})^{K}$ is given by
\[
f^{K}_{1}(\th_{0},\ldots,\th_{n}) = e^{\i\th_{0}} \cdot f( \th_{1}-\th_{0},\ldots,\th_{n}-\th_{0} ).
\]
Since $\Re f = 0$ by assumption, it follows that
\[
\bigl( \Re f^{K}_{1} \bigr)(\th_{0},\ldots,\th_{n}) = - \sin(\th_{0}) \cdot (\Im f)(\th_{1}-\th_{0},\ldots,\th_{n}-\th_{0}).
\]
Hence boundedness of $f$ yields an estimate
\begin{equation} \label{Eqn:TamenessOfu1}
	\left\lvert \int_{0}^{T} \bigl( \Re f^{K}_{1} \bigr)(a_{t}.\th_{0},\ldots,a_{t}.\th_{n}) \, \d t \right\rvert \le \lVert f \rVert_{\infty} \cdot \int_{0}^{T} \lvert \sin(a_{t}.\th_{0}) \rvert \, \d t
\end{equation}
for every $T \in \R$. Recall that the fixed points for the boundary action of $a_{t}$ on $S^{1}$ are $\pm 1$, which in angular coordinates correspond to the multiples of $\pi$. Hence we have
\begin{equation} \label{Eqn:TamenessOfu2}
	\lvert \sin(a_{t}.\th_{0}) \rvert \le \pm \sin(a_{t}.\th_{0})
\end{equation}
for all $t \in \R$, depending on whether $\sin(\th_{0}) \gtreqless 0$. Now with the explicit formula for the operator $\L_{A}$ from \eqref{eqn:FundamentalVectorFields} we compute
\begin{equation} \label{Eqn:TamenessOfu3}
	\int_{0}^{T} \sin(a_{t}.\th_{0}) \, \d t = \int_{0}^{T} \dd{}{t}(a_{t}.\th_{0}) \, \d t = a_{T}.\th_{0} - \th_{0}
\end{equation}
for every $T \in \R$. Combining \eqref{Eqn:TamenessOfu1}, \eqref{Eqn:TamenessOfu2} and \eqref{Eqn:TamenessOfu3} we finally arrive at
\[
\left\lvert \int_{0}^{T} \bigl( \Re f^{K}_{1} \bigr)(a_{t}.\th_{0},\ldots,a_{t}.\th_{n}) \, \d t \right\rvert \le \lVert f \rVert_{\infty} \cdot \lvert a_{T}.\th_{0} - \th_{0} \rvert \le \lVert f \rVert_{\infty} \cdot \pi
\]
for all $T \in \R$ and every point $(\th_{0},\ldots,\th_{n}) \in \T^{n+1}$, which implies that $f^{K}_{1}$ is tame.
\end{proof}

\subsection{The Cauchy problem}
\label{SubSec:CauchyProblem}

Consider the partial differential equation
\begin{equation} \label{eqn:CauchyProblem}
	\L p = u
\end{equation}
with right-hand side $u \in A^{\infty}(\T^{n+1},\C_{1})$. Our goal in this section is to explicitly construct solutions $p \in A(\T^{n+1},\R)$ of this equation, and to study their boundedness properties. As it turns out, solutions of \eqref{eqn:CauchyProblem} are uniquely determined by a suitable choice of initial condition. To formalize this, we make the following definition.

\pagebreak

\begin{definition} \label{def:MeasurableSetOfBasepoints}
	A subset $B_{n} \subset \T^{n+1}$ is called a \emph{measurable set of basepoints} for the boundary action of $G$ on $\T^{n+1}$ if the following two conditions are satisfied.
\begin{enumerate}[leftmargin=1cm,topsep=0.5ex,itemsep=0.5ex]
	\item The set $B_{n}$ is a measurable subset of $\T^{n+1}$.
	\item The map
\[
B_{n} \to \T^{n+1}/G, \quad (b_{0},\ldots,b_{n}) \mapsto G.(b_{0},\ldots,b_{n})
\]
taking each basepoint to its corresponding $G$-orbit in $\T^{n+1}$ is bijective.
\end{enumerate}
\end{definition}

We remark that measurable sets of basepoints $B_{n} \subset \T^{n+1}$ for the boundary action of $G$ on $\T^{n+1}$ as in Definition \ref{def:MeasurableSetOfBasepoints} above exist for every $n \ge 0$ (cf.\,\cite[App.\,B]{Zimmer/Ergodic-theory-and-semisimple-groups}). For any fixed such measurable set of basepoints we may then impose the initial condition
\begin{equation} \label{eqn:InitialCondition}
	p|_{B_{n}} = 0
\end{equation}
upon the solutions of \eqref{eqn:CauchyProblem}. Note that this condition involves pointwise evaluation of the function class $p \in A(\T^{n+1},\R)$ on the set $B_{n} \subset \T^{n+1}$. This is well-defined only on the configuration space $\T^{(n+1)}$, but void on its complement $\T^{n+1} \setminus \T^{(n+1)}$. Nevertheless, as we will see, the initial condition in \eqref{eqn:InitialCondition} uniquely determines the solution $p$. We will refer to \eqref{eqn:CauchyProblem}--\eqref{eqn:InitialCondition} as the \emph{Cauchy problem}. The next proposition characterizes its solutions.

\begin{proposition} \label{prop:CauchyProblem}
	Fix a collection $\Bc = \{B_{n}\}_{n \ge 2}$ of measurable sets of basepoints $B_{n} \subset \T^{n+1}$ for the boundary action of $G$ on $\T^{n+1}$ for all $n \ge 2$. Then there exists a linear operator
\begin{equation} \label{map:SolutionOperatorCauchyProblem}
	\map{\Rop^{n}_{\Bc}}{\im \L^{n}}{A(\T^{n+1},\R)} \quad\quad (n \ge 2)
\end{equation}
which is a right inverse of the Cauchy operator $\L^{n}$ in \eqref{map:OperatorLDownstairs}. More precisely, for every $n \ge 2$ and for every function $u \in A^{\infty}(\T^{n+1},\C_{1})$ satisfying the integrability condition
\begin{equation} \label{eqn:IntegrabilityConditionCauchyProblem}
	\Q u = 0
\end{equation}
the following hold.
\begin{enumerate}[leftmargin=1cm,topsep=0.5ex,itemsep=0.5ex]
	\item Fix a basepoint $\mathbf{b} \in B_{n} \cap \T^{(n+1)}$, an element $g \in G$, and a Cartan decomposition $g = k^{\prime}\,a_{T}\,k$ with $k, k^{\prime} \in K$, $a_{T} \in A$ and $T \in \R$ as in \eqref{eqn:CartanDecomposition}. Then the value of the function $\Rop_{\Bc} u$ at the point $g.\mathbf{b}$ is given by the integral
\begin{equation} \label{eqn:FormulaForRu}
	\bigl( \Rop_{\Bc} u \bigr)( g.\mathbf{b} ) = \int_{0}^{T} (\Re u)(a_{t}\,k.\mathbf{b}) \, \d t.
\end{equation}
	\item The function $p \deq \Rop_{\Bc} u$ is a solution of the Cauchy problem \eqref{eqn:CauchyProblem}--\eqref{eqn:InitialCondition}.
	\item If the function $u$ is tame, then the solution $p = \Rop_{\Bc} u$ is bounded. In particular, the Cauchy-Frobenius complex in \eqref{map:CauchyFrobeniusComplex} is exact at the third term $A_{\t}^{\infty}(\T^{n+1},\C_{1})$.
\end{enumerate}
\end{proposition}

We note that the pointwise evaluation of the function $\Rop_{\Bc} u$ in \eqref{eqn:FormulaForRu} is only defined for points in the configuration space $\T^{(n+1)}$. This is not a loss, however, since we are working with function classes in the sense of Section \ref{subsec:KEquivariantFunctions}.

\begin{proof}
	Fix $n \ge 2$, let $B_{n} \subset \T^{n+1}$ be a measurable set of basepoints, and let $u \in A(\T^{n+1},\C_{1})$ such that \eqref{eqn:IntegrabilityConditionCauchyProblem} holds. Since
\[
\im \L^{n} \subset \ker \Q^{n}
\]
by Proposition \ref{prop:CauchyFrobeniusSequence}, it will be sufficient to explicitly construct the solution $p \in A(\T^{n+1},\R)$ of the Cauchy problem \eqref{eqn:CauchyProblem}--\eqref{eqn:InitialCondition} and to show that it is bounded if $u$ is tame.

\medskip

\noindent{\textbf{Step 1.}} Since the configuration space $\T^{(n+1)}$ is invariant under the action of $G$, we may pick a representative $u \in \As^{\infty}(\T^{n+1},\C_{1})$ such that $u(\z) = 0$ for all $\z \in \T^{n+1} \setminus \T^{(n+1)}$.

\medskip

\noindent{\textbf{Step 2.}} The measurable subset $B_{n} \subset \T^{n+1}$ of basepoints defines a measurable subset $(B_{n})_{K} \subset \T^{n}$ defined by
\[
(B_{n})_{K} \deq \left\{ (b_{1}/b_{0},\ldots,b_{n}/b_{0}) \,\middle|\, (b_{0},\ldots,b_{n}) \in B_{n} \right\}.
\]
In this way, we obtain a bijective parametrization $(B_{n})_{K} \to \T^{n}/P$ of the $P$-orbits in $\T^{n}$.

\medskip

\noindent{\textbf{Step 3.}} We construct a function $q \in \Ss_{P}(\T^{n},\R)$ that solves the Cauchy initial value problem
\begin{equation} \label{eqn:ReducedCauchyProblem}
	\begin{cases} \L q = u_{K}, \\ q|_{(B_{n})_{K}} = 0. \end{cases}
\end{equation}
Here $u_{K} \in \Ss_{P}(\T^{n},\C)$ by Lemma \ref{lemma:ReductionAndExtensionByK}\,(iv), and $(B_{n})_{K} \subset \T^{n}$ is the measurable subset constructed in Step 2. Note that the first equation is obtained from \eqref{eqn:CauchyProblem} by means of $K$-reduction.

\medskip

We will proceed in two stages. First, we solve the initial value problem \eqref{eqn:ReducedCauchyProblem} on the open subset $\To^{(n)} \subset \T^{n}$, which was defined in Section \ref{subsec:BoundaryAction}. We will later extend the solution to all of $\T^{n}$. Writing $u_{K} = u_{K}^{\sharp} + \i u_{K}^{\flat}$ for the decomposition of $u_{K}$ into its real and imaginary parts, we observe that the complex differential equation $\L q = u_{K}$ in \eqref{eqn:ReducedCauchyProblem} is equivalent to the system of real differential equations
\begin{equation} \label{eqn:RealReducedCauchyProblem}
	\L_{A} q = u_{K}^{\sharp}, \quad \L_{N} q = u_{K}^{\flat}.
\end{equation}
Applying Frobenius' theorem (cf.\,\cite[Sec.\,1.3 and Thm.\,1.3.8]{CandelConlon/Foliations.-I}) simultaneously on each $P$-orbit in $\To^{(n)}$, it follows that the system in \eqref{eqn:RealReducedCauchyProblem} admits a $P$-orbitwise smooth solution $q$ on $\To^{(n)}$ if and only if it is involutive (cf.\,\cite{BerhanuCordaroHounie/An-introduction-to-involutive-structures} and \cite[App.\,B]{Hartnick/Bounded-cohomology-via-partial-differential-equations-I}). Note that this argument crucially relies on the facts that $P$ acts freely on $\To^{(n)}$ since $n \ge 2$ by assumption, and that $P$-orbits in $\To^{(n)}$ are connected and simply connected submanifolds of $\To^{(n)}$. The system of differential equations in \eqref{eqn:RealReducedCauchyProblem} is involutive if and only if
\[
[ \L_{A}, \L_{N} ] \, q = \L_{A} u_{K}^{\flat} - \L_{N} u_{K}^{\sharp}.
\]
By the commutator relations from \eqref{eqn:CommutatorRelationsKAN} this amounts to the integrability condition
\[
u_{K}^{\flat} - \L_{A} u_{K}^{\flat} + \L_{N} u_{K}^{\sharp} = 0.
\]
By \eqref{eqn:RealVersionActionOfQ} this is equivalent to
\[
\Q u_{K} = 0.
\]
But this equation is satisfied on $\To^{(n)}$ because $\Q u_{K} = (\Q u)_{K}$ by Lemma \ref{lemma:ReductionAndExtensionByK}\,(iv), and because $\Q u (\z) = 0$ for all $\z \in \T^{(n+1)}$ by \eqref{eqn:IntegrabilityConditionCauchyProblem}. Thus by Frobenius' theorem it follows that the system in \eqref{eqn:RealReducedCauchyProblem} admits a smooth solution $q$ on each $P$-orbit in the open subset $\To^{(n)}$. We may adjust this solution $q$ in such a way that it satisfies the initial condition in \eqref{eqn:ReducedCauchyProblem} on each $P$-orbit in $\To^{(n)}$. Since the subset $(B_{n})_{K} \subset \T^{n}$ is measurable and the right-hand side $u_{K}$ in \eqref{eqn:ReducedCauchyProblem} is a measurable function, it follows that the solution $q$ is a measurable function on $\To^{(n)}$.

It remains to extend the solution $q$ to the whole torus $\T^{n}$. This will be done by setting $q(\z) \deq 0$ for all $\z \in \T^{n} \setminus \To^{(n)}$. Since the complement $\T^{n} \setminus \To^{(n)}$ is of measure zero in $\T^{n}$, since $u_{K}$ vanishes on $\T^{n} \setminus \To^{(n)}$ by Step 1, and since $\To^{(n)}$ is $P$-invariant, this finally yields the desired solution $q \in \Ss_{P}(\T^{n},\R)$ of the Cauchy initial value problem in \eqref{eqn:ReducedCauchyProblem}.

\medskip

\noindent{\textbf{Step 4.}} We prove that the $K$-extension $p \deq q^{K}_{0} \in \Ss_{G}(\T^{n+1},\R)^{K}$ of the function $q$ with weight $0$ is a solution of \eqref{eqn:CauchyProblem}.

\medskip

Applying Lemma \ref{lemma:ReductionAndExtensionByK}\,(iv, i) we deduce from \eqref{eqn:ReducedCauchyProblem} that
\[
(\L p)_{K} = \L p_{K} = \L q = u_{K}.
\]
By Proposition \ref{prop:CauchyFrobeniusSequence} we know that $\L p, u \in \Ss_{G}(\T^{n+1},\C_{1})^{K}$. Hence by Lemma \ref{lemma:ReductionAndExtensionByK}\,(ii) it follows that $\L p = u$.

\medskip

\noindent{\textbf{Step 5.}} We observe that the solution $p \in \Ss_{G}(\T^{n+1},\R)^{K}$ of \eqref{eqn:CauchyProblem} constructed in Step 4 satisfies the initial condition in \eqref{eqn:InitialCondition}.

\medskip

In fact, the solution $q$ of \eqref{eqn:ReducedCauchyProblem} in Step 3 was constructed in such a way that
\[
q( b_{1}/b_{0},\ldots,b_{n}/b_{0} ) = 0
\]
for all basepoints $(b_{0},\ldots,b_{n}) \in B_{n}$. Hence \eqref{eqn:InitialCondition} follows from \eqref{eqn:KExtension} since $p = q^{K}_{0}$ by Step~4.

\medskip

\noindent{\textbf{Step 6.}} We show that $p \in \Ss^{\b}_{G}(\T^{n+1},\R)^{K}$. This proves part (ii) of the proposition.

\medskip

Write $u = u^{\sharp} + \i u^{\flat}$ for the decomposition of $u$ into its real and imaginary parts. By assumption, $\usharp$ and $\uflat$ are bounded functions with bounded $G$-derivatives. By Lemma \ref{lemma:InvarianceAndEquivarianceUnderK} we have $\L_{K} p = 0$. Moreover, $\L p = u$ by Step 4 implies that $\L_{A} p = \usharp$ and $\L_{N} p = \uflat$. It follows that $p$ has bounded $G$-derivatives.

\medskip

\noindent{\textbf{Step 7.}} We derive an explicit formula for the function $p$. This proves part (i) of the proposition.

\medskip

Fix a basepoint $\mathbf{b} \in B_{n}$ and an element $g \in G$. We are going to compute the value of the function $p$ at the point $\z \deq g.\mathbf{b} \in \T^{n+1}$. Choose a Cartan decomposition $g = k^{\prime}\,a_{T}\,k$ with $k, k^{\prime} \in K$ and $a_{T} \in A$ for some $T \in \R$ as in \eqref{eqn:CartanDecomposition}. Then
\begin{equation} \label{Eqn:Computep1}
	p(\z) = p(g.\mathbf{b}) = p(k^{\prime}.(a_{T}\,k).\mathbf{b}) = p(a_{T}.(k.\mathbf{b}))
\end{equation}
since $p$ is $K$-invariant. Taking the real part of the equation $\L p = u$ we obtain $\L_{A} p = \Re u$. Hence by \eqref{eqn:IntegrateAlongA} we have
\begin{equation} \label{Eqn:Computep2}
	p(a_{T}.(k.\mathbf{b})) = p(k.\mathbf{b}) + \int_{0}^{T} (\Re u)(a_{t} \, k.\mathbf{b}) \, \d t.
\end{equation}
Observe that $p(k.\mathbf{b}) = p(\mathbf{b}) = 0$, which follows from $K$-invariance of $p$ and the initial condition in \eqref{eqn:InitialCondition}. Hence combining \eqref{Eqn:Computep1} and \eqref{Eqn:Computep2} we obtain
\begin{equation} \label{Eqn:Computep3}
	p(\z) = \int_{0}^{T} (\Re u)(a_{t} \, k.\mathbf{b}) \, \d t,
\end{equation}
which is the formula in \eqref{eqn:FormulaForRu}. Note that because of the assumption on $u$ in Step 1, the formula in \eqref{Eqn:Computep3} holds for all basepoints in $B_{n}$ including those in the complement of the configuration space.

\medskip

\noindent{\textbf{Step 8.}} Assume that $u$ is tame. Then there exists a constant $C = C(u)$ such that
\[
\left\lvert \int_{0}^{T} (\Re u)(a_{t}.\z) \, \d t \right\rvert < C
\]
for all $\z \in \T^{n+1}$ and $T \in \R$. Hence we conclude from \eqref{Eqn:Computep3} that the solution $p$ is bounded. This proves part (iii) of the proposition.
\end{proof}

\subsection{The Frobenius problem}
\label{SubSec:FrobeniusProblem}

Consider the partial differential equation
\begin{equation} \label{eqn:FrobeniusProblem}
	\Q u = \psi
\end{equation}
with right-hand side $\psi \in A^{\infty}(\T^{n+1},\R)$. Our aim in this section is to explicitly construct a solution $u \in A^{\infty}_{\t}(\T^{n+1},\C_{1})$ of this equation. We will refer to \eqref{eqn:FrobeniusProblem} as the \emph{Frobenius problem}.

\begin{proposition} \label{prop:FrobeniusProblem}
	There exists a linear operator
\begin{equation} \label{map:SolutionOperatorFrobeniusProblem}
	\map{\Sop^{n}}{A^{\infty}(\T^{n+1},\R)}{A^{\infty}_{\t}(\T^{n+1},\C_{1})} \quad\quad (n \ge 1)
\end{equation}
which is a right inverse of the Frobenius operator $\Q^{n}$ in \eqref{map:OperatorQDownstairs}. More precisely, for every $n \ge 1$ and for every function $\psi \in A^{\infty}(\T^{n+1},\R)$ the following hold.
\begin{enumerate}[leftmargin=1cm,topsep=0.5ex,itemsep=0.5ex]
	\item The value of the function $\Sop \psi \in A^{\infty}_{\t}(\T^{n+1},\C_{1})$ at any point $(z_{0},\ldots,z_{n}) \in \T^{n+1}$ is given by the integral
\begin{equation} \label{eqn:FormulaForSpsi}
	(\Sop \psi)(z_{0},\ldots,z_{n}) = \i \cdot z_{0} \cdot \int_{0}^{\infty} \psi\bigl( 1,a_{t}.(z_{1}/z_{0}),\ldots,a_{t}.(z_{n}/z_{0}) \bigr) \cdot e^{-t} \, \d t.
\end{equation}
	\item The function $u \deq \Sop \psi$ is a solution of the Frobenius problem \eqref{eqn:FrobeniusProblem}. In particular, the Cauchy-Frobenius complex in \eqref{map:CauchyFrobeniusComplex} is exact at the fourth term $A^{\infty}(\T^{n+1},\R)$.
\end{enumerate}
\end{proposition}

\begin{proof}
	Fix $n \ge 1$ and let $\psi \in A(\T^{n+1},\R)$. To prove the proposition, it will be sufficient to construct an explicit solution $u \in A^{\infty}_{\t}(\T^{n+1},\C_{1})$ of the Frobenius problem \eqref{eqn:FrobeniusProblem}.

\medskip

\noindent{\textbf{Step 1.}} Pick a representative $\psi \in \As^{\infty}(\T^{n+1},\R)$.

\medskip

\noindent{\textbf{Step 2.}} Observe that $\psi_{K} \in \Ls^{\infty}(\T^{n},\R)$ by Lemma \ref{lemma:ReductionAndExtensionByK}\,(iii). We define a measurable function $v \in \Ls^{0}(\T^{n},\C)$ by
\begin{equation} \label{eqn:DefinitionOfv}
	v(\z) \deq \i \cdot \int_{0}^{\infty} \psi_{K}(a_{s}.\z) \cdot e^{-s} \, \d s
\end{equation}
for all $\z \in \T^{n}$.

\medskip

\noindent{\textbf{Step 3.}} We prove that $v$ is a bounded function contained in $\Ss^{\b}_{P}(\T^{n},\C)$.

\medskip

We have seen in Step 2 that $\psi_{K}$ is bounded. It follows that
\begin{equation} \label{eqn:BoundednessOfv}
	\lvert v(\z) \rvert \le \lVert \psi_{K} \rVert_{\infty} \cdot \int_{0}^{\infty} e^{-s} \, \d s = \lVert \psi_{K} \rVert_{\infty}
\end{equation}
for every $\z \in \T^{n}$, which implies that $v$ is bounded. Next we observe that $\psi_{K} \in \Ss^{\b}_{P}(\T^{n},\R)$ by Lemma \ref{lemma:ReductionAndExtensionByK}\,(v). We are going to show that $v$ is $P$-orbitwise smooth with bounded $P$-derivatives. For $s \ge 0$ consider the function $f_{s} \in \Ls^{\infty}(\T^{n},\R)$ defined by
\[
f_{s}(\z) \deq \psi_{K}(a_{s}.\z).
\]
It is $P$-orbitwise smooth since the map
\begin{equation} \label{map:SmoothnessAndBoundednessOfv}
	P \to \R, \quad p \mapsto f_{s}(p.\z) = \psi_{K}(a_{s}\,p.\z)
\end{equation}
is smooth for every $\z \in \T^{n}$ because $a_{s}\,p \in P$ and $\psi_{K}$ is $P$-orbitwise smooth. Now for every $\z \in \T^{n}$ we compute
\[
\begin{aligned}
	(\L_{A} f_{s})(\z) &= \left. \dd{}{t} \right|_{t=0} f_{s}(a_{t}.\z) \, = \left. \dd{}{t} \right|_{t=0} \psi_{K}(a_{s}.(a_{t}.\z)) \\
	&= \left. \dd{}{t} \right|_{t=0} \psi_{K}(a_{t}.(a_{s}.\z)) \, = (\L_{A} \psi_{K})(a_{s}.\z)
\end{aligned}
\]
and, using the relation $a_{s}.n_{t} = n_{e^{-s} \cdot t}.a_{s}$ from \eqref{eqn:ANormalizesN},
\[
\begin{aligned}
	(\L_{N} f_{s})(\z) &= \left. \dd{}{t} \right|_{t=0} f_{s}(n_{t}.\z) \, = \left. \dd{}{t} \right|_{t=0} \psi_{K}(a_{s}.(n_{t}.\z)) \\
	&= \left. \dd{}{t} \right|_{t=0} \psi_{K}(n_{e^{-s} \cdot t}.(a_{s}.\z)) \, = e^{-s} \cdot (\L_{N} \psi_{K})(a_{s}.\z).
\end{aligned}
\]
Since $\psi_{K}$ has bounded $P$-derivatives and $s \ge 0$, we conclude that $\L_{A}f_{s}$ and $\L_{N}f_{s}$ are both bounded. Hence by an estimate as in \eqref{eqn:BoundednessOfv} above, by \eqref{eqn:DefinitionOfv} the Lebesgue dominated convergence theorem implies that the derivatives $\L_{A}v$ and $\L_{N}v$ exist and are bounded. Since $\psi_{K}$ has bounded $P$-derivatives, a similar argument involving the derivatives $\L_{i_{1},\ldots,i_{\ell}} f_{s}$ for all integers $\ell > 0$ and all $(i_{1},\ldots,i_{\ell}) \in \{A,N\}^{\ell}$ shows that the function $v$ has bounded $P$-derivatives.

\medskip

\noindent{\textbf{Step 4.}} We show that the function $v$ is a solution of the differential equation
\begin{equation} \label{eqn:ReducedFrobeniusProblem1}
	\Q v = \psi_{K},
\end{equation}
which is obtained from \eqref{eqn:FrobeniusProblem} by means of $K$-reduction.

\medskip

Recall that $v = \vsharp + \i \, \vflat$ denotes the decomposition of the complex function $v$ into its real and imaginary parts. We see from \eqref{eqn:DefinitionOfv} that $\vsharp = 0$, hence we obtain
\[
\Q v = \vflat - \L_{A} \vflat + \L_{N} \vsharp = \vflat - \L_{A} \vflat
\]
by \eqref{eqn:RealVersionActionOfQ}. Thus \eqref{eqn:ReducedFrobeniusProblem1} turns out to be equivalent to
\begin{equation} \label{eqn:ReducedFrobeniusProblem2}
	\vflat - \L_{A} \vflat = \psi_{K}.
\end{equation}
We know from Step 3 that the derivative $\L_{A} \vflat$ exists. Hence by the Lebesgue dominated convergence theorem, for every $\z \in \T^{n}$ we compute
\[
\begin{aligned}
	(\L_{A} \vflat) (\z) &= \left. \dd{}{t}\right|_{t=0} \vflat(a_{t}.\z) \,=\, \left. \dd{}{t}\right|_{t=0} \int_{0}^{\infty} \psi_{K}(a_{t}.(a_{s}.\z)) \cdot e^{-s} \, \d s \\
	&= \int_{0}^{\infty} \left( \left. \dd{}{t}\right|_{t=0} \psi_{K}(a_{t+s}.\z) \right) \cdot e^{-s} \, \d s \, = \, \int_{0}^{\infty} \left( \dd{}{s} \psi_{K}(a_{s}.\z) \right) \cdot e^{-s} \, \d s \\
	&= \left[ \psi_{K}(a_{s}.\z) \cdot e^{-s} \right]_{0}^{\infty} + \int_{0}^{\infty} \psi_{K}(a_{s}.\z) \cdot e^{-s} \, \d s \,=\, - \psi_{K}(\z) + \vflat(\z).
\end{aligned}
\]
Here the second last identity holds by integration by parts. Hence $\vflat$ is a solution of \eqref{eqn:ReducedFrobeniusProblem2}.

\medskip

\noindent{\textbf{Step 5.}} We prove that the $K$-extension $u = v^{K}_{1} \in \As^{\infty}(\T^{n+1},\C_{1})$ of the function $v$ with weight $1$ is a solution of \eqref{eqn:FrobeniusProblem}. Together with \eqref{eqn:DefinitionOfv} and \eqref{eqn:KExtension} this implies part (i) of the proposition.

\medskip

Applying Lemma \ref{lemma:ReductionAndExtensionByK}\,(iv, i) we deduce from \eqref{eqn:ReducedFrobeniusProblem1} that
\[
(\Q u)_{K} = \Q u_{K} = \Q v = \psi_{K}.
\]
By Proposition \ref{prop:CauchyFrobeniusSequence} we know that $\Q u, \psi \in \Ss_{G}(\T^{n+1},\R)^{K}$. Hence by Lemma \ref{lemma:ReductionAndExtensionByK}\,(ii) it follows that $\Q u = \psi$.

\medskip

\noindent{\textbf{Step 6.}} We see from \eqref{eqn:DefinitionOfv} that $\Re v = 0$. Hence by Lemma \ref{lemma:CriterionForTameness} the $K$-extension $u = v^{K}_{1}$ is tame and therefore defines a function $u \in A^{\infty}_{\t}(\T^{n+1},\C_{1})$. This proves part (ii) of the proposition.
\end{proof}

\section{Transgression}
\label{sec:Transgression}

\subsection{The transgression map}
\label{subsec:TransgressionMap}

Let us begin with the following basic observation.

\begin{lemma} \label{lemma:OperatorsLAndQIntertwine}
	The differential operators $\L^{n}$ and $\Q^{n}$ in \eqref{map:OperatorLDownstairs} and \eqref{map:OperatorQDownstairs} satisfy the relations
\[
\L^{n+1} \circ \de^{n} = \de^{n} \circ \L^{n} \quad \text{and} \quad \Q^{n+1} \circ \de^{n} = \de^{n} \circ \Q^{n}
\]
for every $n \ge 0$. They therefore define cochain maps
\[
\map{\L^{n}}{A(\T^{n+1},\R)}{A^{\infty}(\T^{n+1},\C_{1})} \quad\quad (n \ge 0)
\]
and
\[
\map{\Q^{n}}{A^{\infty}(\T^{n+1},\C_{1})}{A^{\infty}(\T^{n+1},\R)} \quad\quad (n \ge 0).
\]
\end{lemma}

\begin{proof}
	By \cite[Lemma 3.3]{Hartnick/Bounded-cohomology-via-partial-differential-equations-I} the action of the differential operators $\L_{K}$, $\L_{A}$ and $\L_{N}$ on orbitwise smooth functions intertwines with the action of the homogeneous coboundary operator $\de$. Hence the claim follows from Definitions \ref{def:OperatorL} and \ref{def:OperatorQ}.
\end{proof}

By the lemma, the Cauchy-Frobenius complex in \eqref{map:CauchyFrobeniusComplex} gives rise to a double complex
\[
\begin{tikzcd}[column sep = scriptsize]
	& \vdots \arrow[d,rightarrow]
	& \vdots \arrow[d,rightarrow]
	& \vdots \arrow[d,rightarrow]
	& \vdots \arrow[d,rightarrow] \\
	0 \arrow[r,rightarrow]
	& A^{\infty}(\T^{n-1},\R)^{P} \arrow[r,"\iota^{n-2}",rightarrow] \arrow[d,"\de^{n-2}",rightarrow]
	& A^{\infty}(\T^{n-1},\R) \arrow[r,"\L^{n-2}",rightarrow] \arrow[d,"\de^{n-2}",rightarrow]
	& A^{\infty}_{\t}(\T^{n-1},\C_{1}) \arrow[r,"\Q^{n-2}",rightarrow] \arrow[d,"\de^{n-2}",rightarrow]
	& A^{\infty}(\T^{n-1},\R) \arrow[r,rightarrow] \arrow[d,"\de^{n-2}",rightarrow] 
	& 0 \\
	0 \arrow[r,rightarrow]
	& A^{\infty}(\T^{n},\R)^{P} \arrow[r,"\iota^{n-1}",rightarrow] \arrow[d,"\de^{n-1}",rightarrow]
	& A^{\infty}(\T^{n},\R) \arrow[r,"\L^{n-1}",rightarrow] \arrow[d,"\de^{n-1}",rightarrow]
	& A^{\infty}_{\t}(\T^{n},\C_{1}) \arrow[r,"\Q^{n-1}",rightarrow] \arrow[d,"\de^{n-1}",rightarrow]
	& A^{\infty}(\T^{n},\R) \arrow[r,rightarrow] \arrow[d,"\de^{n-1}",rightarrow]
	& 0 \\
	0 \arrow[r,rightarrow]
	& A^{\infty}(\T^{n+1},\R)^{P} \arrow[r,"\iota^{n}",rightarrow] \arrow[d,"\de^{n}",rightarrow]
	& A^{\infty}(\T^{n+1},\R) \arrow[r,"\L^{n}",rightarrow] \arrow[d,"\de^{n}",rightarrow]
	& A^{\infty}_{\t}(\T^{n+1},\C_{1}) \arrow[r,"\Q^{n}",rightarrow] \arrow[d,"\de^{n}",rightarrow]
	& A^{\infty}(\T^{n+1},\R) \arrow[r,rightarrow] \arrow[d,"\de^{n}",rightarrow]
	& 0 \\
	& \vdots & \vdots & \vdots & \vdots
\end{tikzcd}
\]
with commuting differentials. Abbreviating the vertical complexes by $\Ac^{\infty}_{P}$, $\Ac^{\infty}$ and $\Ac^{\infty}_{\t}$, where
\[
\begin{aligned}
	(\Ac^{\infty}_{P})^{n} &\deq ( A^{\infty}(\T^{n+1},\R)^{P},\de^{n} ), \\
	(\Ac^{\infty})^{n} &\deq ( A^{\infty}(\T^{n+1},\R),\de^{n} ), \\
	(\Ac^{\infty}_{\t})^{n} &\deq ( A^{\infty}_{\t}(\T^{n+1},\C_{1}),\de^{n} )
\end{aligned}
\]
for all $n \ge 0$, we may write this double complex more conveniently as
\begin{equation} \label{map:DoubleComplex}
	\begin{tikzcd}
		0 \arrow[r,rightarrow] & \Ac^{\infty}_{P} \arrow[r,"\iota",rightarrow]
		& \Ac^{\infty} \arrow[r,"\L",rightarrow]
		& \Ac^{\infty}_{\t} \arrow[r,"\Q",rightarrow]
		& \Ac^{\infty} \arrow[r,rightarrow] & 0.
	\end{tikzcd}
\end{equation}
Define a subcomplex $\Ec \subset \Ac^{\infty}_{\t}$ by
\[
\Ec^{n} \deq \ker \Q^{n} = \im \L^{n}
\]
for all $n \ge 0$, and denote by $\map{i^{n}}{\Ec^{n}}{(\Ac^{\infty}_{\t})^{n}}$ the canonical inclusion. The sequence in \eqref{map:DoubleComplex} then splits into the short sequences
\begin{subequations}
\begin{equation} \label{map:ShortExactSequenceL}
	\begin{tikzcd}
		0 \arrow[r,rightarrow] & (\Ac^{\infty}_{P})^{n} \arrow[r,"\iota^{n}",rightarrow]
		& (\Ac^{\infty})^{n} \arrow[r,"\L^{n}",rightarrow]
		& \Ec^{n} \arrow[r,rightarrow] & 0  \quad\quad (n \ge 0)
	\end{tikzcd}
\end{equation}
and
\begin{equation} \label{map:ShortExactSequenceQ}
	\begin{tikzcd}
		0 \arrow[r,rightarrow] & \Ec^{n} \arrow[r,"i^{n}",rightarrow]
		& (\Ac^{\infty}_{\t})^{n} \arrow[r,"\Q^{n}",rightarrow]
		& (\Ac^{\infty})^{n} \arrow[r,rightarrow] & 0 \quad\quad (n \ge 0).
	\end{tikzcd}
\end{equation}
\end{subequations}
From the exactness properties of the sequence in \eqref{map:DoubleComplex} we then obtain long exact sequences in cohomology, as follows.

\begin{lemma} \label{lemma:LongExactSequences}
	There are long exact sequences
\begin{subequations}
\begin{equation} \label{map:LongExactSequenceL}
	\begin{tikzcd}[column sep = scriptsize]
		H^{2}(\Ac_{P}^{\infty}) \arrow[r,"\iota^{\ast}"] & H^{2}(\Ac^{\infty}) \arrow[r,"\L^{\ast}"] & H^{2}(\Ec) \arrow[r,"\Phi_{\L}^{2}"] & H^{3}(\Ac_{P}^{\infty}) \arrow[r,"\iota^{\ast}"] & \cdots \hspace{1.4cm} \\
		\hspace{5mm} \cdots \arrow[r,"\iota^{\ast}"] & H^{n}(\Ac^{\infty}) \arrow[r,"\L^{\ast}"] & H^{n}(\Ec) \arrow[r,"\Phi_{\L}^{n}"] & H^{n+1}(\Ac_{P}^{\infty}) \arrow[r,"\iota^{\ast}"] & H^{n+1}(\Ac^{\infty}) \arrow[r] & \cdots
	\end{tikzcd}
\end{equation}
and
\begin{equation} \label{map:LongExactSequenceQ}
	\begin{tikzcd}[column sep = scriptsize]
		H^{1}(\Ec) \arrow[r,"i^{\ast}"] & H^{1}(\Ac_{\t}^{\infty}) \arrow[r,"\Q^{\ast}"] & H^{1}(\Ac^{\infty}) \arrow[r,"\Phi_{\Q}^{1}"] & H^{2}(\Ec) \arrow[r,"i^{\ast}"] & \cdots \hspace{1.5cm} \\
		\hspace{5mm} \cdots \arrow[r,"i^{\ast}"] & H^{n}(\Ac_{\t}^{\infty}) \arrow[r,"\Q^{\ast}"] & H^{n}(\Ac^{\infty}) \arrow[r,"\Phi_{\Q}^{n}"] & H^{n+1}(\Ec) \arrow[r,"i^{\ast}"] & H^{n+1}(\Ac_{\t}^{\infty}) \arrow[r,] & \cdots
	\end{tikzcd}
\end{equation}
\end{subequations}
with connecting homomorphisms
\begin{subequations}
\begin{equation} \label{map:ConnectingHomomorphismL}
	\map{\Phi_{\L}^{n}}{H^{n}(\Ec)}{H^{n+1}(\Ac_{P}^{\infty})} \quad\quad (n \ge 2)
\end{equation}
and
\begin{equation} \label{map:ConnectingHomomorphismQ}
	\map{\Phi_{\Q}^{n}}{H^{n}(\Ac^{\infty})}{H^{n+1}(\Ec)} \quad\quad (n \ge 1).
\end{equation}
\end{subequations}
\end{lemma}

\begin{proof}
	By Proposition \ref{prop:ExactnessOfCauchyFrobeniusComplex} the sequence in \eqref{map:ShortExactSequenceL} is exact for every $n \ge 2$, and hence gives rise to the long exact sequence in \eqref{map:LongExactSequenceL}. Likewise, by Proposition \ref{prop:ExactnessOfCauchyFrobeniusComplex} the sequence in \eqref{map:ShortExactSequenceQ} is exact for every $n \ge 1$ and thus gives rise to the long exact sequence in \eqref{map:LongExactSequenceQ}.
\end{proof}

We denote the composition of the homomorphism in \eqref{map:EquivariantLiftingOnCohomology} with the isomorphism in \eqref{map:BoundaryModel} by
\begin{equation} \label{map:HomomorphismEquivariantLift}
	\map{\Pi^{n}}{H^{n}(\Ac_{P}^{\infty})}{H^{n}_{\cb}(G;\R)} \quad\quad (n \ge 0).
\end{equation}
This map is surjective by Proposition \ref{prop:EquivariantLifting}. It will be called the \emph{lifting homomorphism}.

\begin{definition}
	The concatenation of the lifting homomorphism in \eqref{map:HomomorphismEquivariantLift} with the connecting homomorphisms in \eqref{map:ConnectingHomomorphismL} and \eqref{map:ConnectingHomomorphismQ} defines the \emph{transgression map}
\begin{equation} \label{map:TransgressionMap}
	\map{\Lambda^{n} \deq \Pi^{n} \circ \Phi_{\L}^{n-1} \circ \Phi_{\Q}^{n-2}}{H^{n-2}(\Ac^{\infty})}{H^{n}_{\cb}(G;\R}) \quad\quad (n > 2).
\end{equation}
\end{definition}

Notice that the transgression map is defined for all degrees $n>2$, and that it shifts the degree by $2$.

\subsection{Transgressive classes}
\label{subsec:TransgressiveClasses}

We characterize classes in the image of the transgression map.

\begin{definition}
	A class $\alpha \in H^{n}_{\cb}(G;\R)$ of degree $n>2$ is called \emph{transgressive} if it is contained in the image of the transgression map $\Lambda^{n}$ in \eqref{map:TransgressionMap}.
\end{definition}

\begin{proposition} \label{prop:VanishingForTransgressiveClasses}
	Let $\alpha \in H^{n}_{\cb}(G;\R)$ with $n>2$. If $\alpha$ is transgressive, then $\alpha = 0$.
\end{proposition}

\begin{proof}
	This is immediate since $H^{n-2}(\Ac^{\infty}) = 0$ by Corollary \ref{cor:VanishingForA} for every $n>2$.
\end{proof}

We next derive a useful criterion that helps to decide whether a given bounded cohomology class is transgressive. To this end, we first recall the vanishing $H^{n}(\Ac^{\infty}) = 0$ which holds for every $n>0$ by Corollary \ref{cor:VanishingForA}. Exactness of the long sequence in \eqref{map:LongExactSequenceL} therefore implies that the connecting homomorphism $\Phi_{\L}^{n-1}$ in \eqref{map:ConnectingHomomorphismL} is in fact an isomorphism for every $n>2$. Consider now the diagram
\[
\begin{tikzcd}[column sep = large, row sep = large]
	H^{n-1}(\Ec) \arrow[d,"i^{\ast}",rightarrow] \arrow[r,"\Phi_{\L}^{n-1}",rightarrow]  & H^{n}(\Ac_{P}^{\infty}) \arrow[r,"\Pi^{n}",rightarrow] & H^{n}_{\cb}(G;\R) \\
	H^{n-1}(\Ac_{\t}^{\infty})
\end{tikzcd}
\]
It gives rise to the following cohomological characterization of transgressive classes.

\begin{proposition} \label{prop:CriterionForTransgressive}
	A class $\alpha \in H^{n}_{\cb}(G;\R)$ with $n>2$ is transgressive if and only if there exists a class $\beta \in H^{n}(\Ac_{P}^{\infty})$ such that
\begin{equation} \label{eqn:CriterionForTransgressive}
	i^{\ast} \circ (\Phi_{\L}^{n-1})^{-1} \, \beta = 0 \quad \text{and} \quad \Pi^{n} \, \beta = \alpha.
\end{equation}
\end{proposition}

\begin{proof}
	Fix $n>2$, and consider a class $\alpha \in H^{n}_{\cb}(G;\R)$. If $\alpha$ is transgressive, then $\alpha=0$ by Proposition \ref{prop:VanishingForTransgressiveClasses} above and hence the class $\beta=0$ satisfies the conditions in \eqref{eqn:CriterionForTransgressive}. For the converse, assume that there exists $\beta \in H^{n}(\Ac_{P}^{\infty})$ such that \eqref{eqn:CriterionForTransgressive} holds. Since $\Phi_{\L}^{n-1}$ is an isomorphism, there is $\nu \in H^{n-1}(\Ec)$ such that $\Phi_{\L}^{n-1} \nu = \beta$ and $i^{\ast} \nu = 0$. Hence exactness of the long sequence in \eqref{map:LongExactSequenceQ} implies that there exists $\omega \in H^{n-2}(\Ac^{\infty})$ such that $\Phi_{\Q}^{n} \, \omega = \nu$. It follows that $\Lambda^{n-2} \, \omega = \alpha$ and hence $\alpha$ is transgressive.
\end{proof}

\subsection{Reducible classes}
\label{subsec:ReducibleClasses}

Recall from Section \ref{subsec:ContinuousBoundedCohomology} that the continuous bounded cohomology of $G$ is endowed with a natural cup product
\begin{equation} \label{map:CupProductOnCohomOfG}
	\map{\smallsmile}{H^{n}_{\cb}(G;\R) \otimes H^{m}_{\cb}(G;\R)}{H^{n+m}_{\cb}(G;\R)}.
\end{equation}
We may define a similar cup product on the cohomology of the complex $\Ac_{P}^{\infty}$, as follows. Consider first the cup product
\[
\map{\cup}{\Ls^{0}(\T^{n+1},\C) \otimes \Ls^{0}(\T^{m+1},\C)}{\Ls^{0}(\T^{n+m+1},\C)} \quad\quad (n,m \ge 0)
\]
on the homogeneous bar complex \eqref{map:MeasurableCochainComplex} of measurable cochains, which for $f \in \Ls^{0}(\T^{n+1},\C)$ and $g \in \Ls^{0}(\T^{m+1},\C)$ is defined by
\begin{equation} \label{eqn:DefinitionOfCupProduct}
	(f \cup g)(z_{0},\ldots,z_{n+m}) \deq f(z_{0},\ldots,z_{n}) \cdot g(z_{n},z_{n+1},\ldots,z_{n+m}).
\end{equation}
It gives rise to cup products
\begin{equation} \label{map:CupProductOnAinftyReal}
	\map{\cup}{A^{\infty}(\T^{n+1},\R)^{P} \otimes A^{\infty}(\T^{m+1},\R)^{P}}{A^{\infty}(\T^{n+m+1},\R)^{P}} \quad\quad (n,m \ge 0)
\end{equation}
and
\begin{equation} \label{map:CupProductOnAinftyComplex}
	\map{\cup}{A^{\infty}(\T^{n+1},\C_{\m}) \otimes A^{\infty}(\T^{m+1},\R)}{A^{\infty}(\T^{n+m+1},\C_{\m})} \quad\quad (n,m \ge 0)
\end{equation}
for every $\m \in \Z$. The former product in \eqref{map:CupProductOnAinftyReal} then induces a corresponding cup product
\begin{equation} \label{map:CupProductOnCohomOfA^P}
	\map{\cup}{H^{n}(\Ac_{P}^{\infty}) \otimes H^{m}(\Ac_{P}^{\infty})}{H^{n+m}(\Ac_{P}^{\infty})}
\end{equation}
on the cohomology of the complex $\Ac_{P}^{\infty}$ \cite[Sec.\,2]{Bucher-Karlsson/The-simplicial-volume-of-closed-manifolds-covered-by-Bbb-Hsp-2timesBbb-Hsp-2}.

\begin{lemma} \label{lemma:CupProductIntertwinesWithLifting}
	The cup products $\cup$ and $\smallsmile$ in \eqref{map:CupProductOnCohomOfA^P} and \eqref{map:CupProductOnCohomOfG} intertwine with the lifting homomorphism $\Pi^{n}$ in \eqref{map:HomomorphismEquivariantLift}.
\end{lemma}

\begin{proof}
	This follows from the naturality of the cochain map in \eqref{map:EquivariantLiftingFromA^P}, in combination with the fact that the isomorphism in \eqref{map:BoundaryModel} is compatible with the ring structure on cohomology determined by the cup products $\smallsmile$ and $\cup$ \cite[Thm.\,7.5.3]{Monod/Continuous-bounded-cohomology-of-locally-compact-groups}.
\end{proof}

\begin{lemma} \label{lemma:OperatorsActingOnCupProduct}
	Fix an integer $\m \in \Z$, let $n \ge 1 $ and $m \ge 0$, and let $f \in A^{\infty}(\T^{n+1},\C_{\m})$ and $g \in A^{\infty}(\T^{m+1},\R)$. Then the cup product in \eqref{map:CupProductOnAinftyComplex} has the following properties.
\begin{enumerate}[leftmargin=1cm,topsep=0.5ex,itemsep=0.5ex]
	\item $\L^{n+m} (f \cup g) = (\L^{n} f) \cup g + f \cup (\L^{m} g)$;
	\item $\I^{n+m} (f \cup g) = (\I^{n} f) \cup g$;
	\item $(f \cup g)_{K} = f_{K} \cup g$.
\end{enumerate}
\end{lemma}

\begin{proof}
	Let $f \in A^{\infty}(\T^{n+1},\C_{\m})$ and $g \in A^{\infty}(\T^{m+1},\R)$ with $n \ge 1 $ and $m \ge 0$, and consider the formula for the product $f \cup g$ in \eqref{eqn:DefinitionOfCupProduct}. We see from Definition \ref{def:OperatorL} that (i) is a consequence of the product rule for differentiable functions, while (ii) is immediate from \eqref{eqn:DefinitionOfOperatorI}. The formula in (iii) follows from \eqref{eqn:KReduction}.
\end{proof}

\begin{definition}
	A bounded cohomology class $\alpha \in H^{n}_{\cb}(G;\R)$ of degree $n>2$ is called \emph{strongly reducible} if it admits a product decomposition
\begin{equation} \label{eqn:ProductDecomposition}
	\alpha = \alpha^{\prime} \smallsmile \alpha^{\prime\prime}
\end{equation}
with factors $\alpha^{\prime} \in H^{2}_{\cb}(G;\R)$ and $\alpha^{\prime\prime} \in H^{n-2}_{\cb}(G;\R)$.
\end{definition}

Let us make this definition more concrete. To this end, we recall from Section \ref{subsec:ContinuousBoundedCohomology} that the second bounded cohomology $H^{2}_{\cb}(G;\R) \cong \R$ is generated by the bounded K\"ahler class $\kappa \in H^{2}_{\cb}(G;\R)$. Hence the first factor $\alpha^{\prime}$ in the product decomposition in \eqref{eqn:ProductDecomposition} is in fact a real multiple of $\kappa$. For our purposes in this section, we will further need to know that under the isomorphism in \eqref{map:BoundaryModel}, the bounded K\"ahler class $\kappa \in H^{2}_{\cb}(G;\R)$ is identified with the cohomology class of the \emph{orientation cocycle} $\Or \in L^{\infty}(\T^{3},\R)^{G}$ \cite[Sec.\,2.3]{BurgerIozzi/A-useful-formula-from-bounded-cohomology}. This latter cocycle is defined by
\begin{equation} \label{eqn:OrientationCocycle}
	\Or(z_{0},z_{1},z_{2}) \deq
	\begin{cases}
		1 & \text{if $(z_{0},z_{1},z_{2})$ is positively oriented;} \\
		-1 & \text{if $(z_{0},z_{1},z_{2})$ is negatively oriented;} \\
		0 & \text{otherwise}
	\end{cases}
\end{equation}
for all triples $(z_{0},z_{1},z_{2}) \in \T^{3}$ of points on $S^{1}$. The orientation cocycle naturally defines a cocycle in $A^{\infty}(\T^{3},\R)^{P}$ as well. This cocycle is given by the same formula as in \eqref{eqn:OrientationCocycle} and will be denoted by the same symbol $\Or$. We are now ready to prove the following sufficient criterion for a class to be transgressive.

\begin{proposition} \label{prop:StronlgyReducibleImpliesTransgressive}
	Let $\alpha \in H^{n}_{\cb}(G;\R)$ with $n>2$. If $\alpha$ is strongly reducible, then $\alpha$ is transgressive.
\end{proposition}

The proof of the proposition relies on the following lemma.

\begin{lemma} \label{lemma:FunctionILIOr}
	The function $\I \L \I \Or \in A^{\infty}(\T^{1},\C_{1})$ is given by
\[
(\I \L \I \Or)(z) = \frac{i}{\pi} \cdot z.
\]
\end{lemma}

\begin{proof}
	Borrowing \eqref{eqn:FormulaL_AI} and \eqref{eqn:FormulaL_NI} from the proof of Proposition \ref{prop:ComplexAIsAcyclic}, we infer that in angular coordinates, the function $\I \L \I \Or \in A^{\infty}(\T^{1},\C_{1})$ is expressed by the integral
\[
( \I \L \I \Or )(\th) = \frac{1}{4\pi^{2}} \, \int_{0}^{2\pi} \int_{0}^{2\pi} e^{i \eta} \cdot \Or(\eta,\varphi,\th) \, \d\eta \, \d\varphi.
\]
Here we used that the orientation cocycle is $G$-invariant. Then it is an exercise to compute from the explicit formula in \eqref{eqn:OrientationCocycle} that this integral equals $(i/\pi) \cdot e^{i \th}$.
\end{proof}

\begin{proof}[Proof of Proposition \ref{prop:StronlgyReducibleImpliesTransgressive}]
	Fix $n>2$, and consider a strongly reducible class
\[
\alpha = \alpha^{\prime} \smallsmile \alpha^{\prime\prime}
\]
with $\alpha^{\prime} \in H^{2}_{\cb}(G;\R)$ and $\alpha^{\prime\prime} \in H^{n-2}_{\cb}(G;\R)$. We are going to show that $\alpha$ satisfies the criterion in Proposition \ref{prop:CriterionForTransgressive}. This is trivially true if $\alpha^{\prime} = 0$, hence we will assume that $\alpha^{\prime} \neq 0$. For ease of notation, we will mostly suppress the canonical inclusions $\iota$ and $i$ throughout this proof. 

\medskip

\noindent{\textbf{Step 1.}} Since the lifting homomorphism in \eqref{map:HomomorphismEquivariantLift} is surjective, there exist classes $\beta^{\prime} \in H^{2}(\Ac_{P}^{\infty})$ and $\beta^{\prime\prime} \in H^{n-2}(\Ac_{P}^{\infty})$ such that
\[
\Pi^{2} \, \beta^{\prime} = \alpha^{\prime} \quad \text{and} \quad \Pi^{n-2} \, \beta^{\prime\prime} = \alpha^{\prime\prime}.
\]

\medskip

\noindent{\textbf{Step 2.}} Recall from the above that $H^{2}_{\cb}(G;\R) \cong \R$, with an explicit generator determined by the orientation cocycle $\Or \in L^{\infty}(\T^{3},\R)^{G}$ via the isomorphism in \eqref{map:BoundaryModel}. We think of the orientation cocycle as an element of $A^{\infty}(\T^{3},\R)^{P}$. Since $\alpha^{\prime} \neq 0$ by assumption, it follows that $\alpha^{\prime}$ is in fact a real multiple of $\Pi^{2} \, [\Or]$. Hence, rescaling $\alpha^{\prime\prime}$ and $\beta^{\prime\prime}$ with the same factor if necessary, we may without loss of generality assume that
\begin{equation} \label{eqn:BetaPrimeIsOr}
	\beta^{\prime} = [\Or].
\end{equation}

\medskip

\noindent{\textbf{Step 3.}} Consider now the product $\beta \deq \beta^{\prime} \cup \beta^{\prime\prime} \in H^{n}(\Ac_{P}^{\infty})$. It follows with Lemma \ref{lemma:CupProductIntertwinesWithLifting} and Step 1 that
\[
\Pi^{n} \, \beta = \Pi^{2} \, \beta^{\prime} \smallsmile \Pi^{n-2} \, \beta^{\prime\prime} = \alpha^{\prime} \smallsmile \alpha^{\prime\prime} = \alpha.
\]

\medskip

\noindent{\textbf{Step 4.}} We pick a cocycle $b^{\prime\prime} \in (\Ac_{P}^{\infty})^{n-2}$ representing the class $\beta^{\prime\prime}$. Then \eqref{eqn:BetaPrimeIsOr} implies that the cocycle
\begin{equation} \label{eqn:ProductCocycleb}
	b \deq \Or \cup b^{\prime\prime} \in (\Ac_{P}^{\infty})^{n}
\end{equation}
is a representative for the class $\beta$.

\medskip

\noindent{\textbf{Step 5.}} We claim that the cocycle $u \in (\Ac_{\t}^{\infty})^{n-1}$ defined by
\[
u \deq \L \, \I \, b
\]
represents the class $i^{\ast} \circ (\Phi_{\L}^{n-1})^{-1} \, \beta \in H^{n-1}(\Ac_{\t}^{\infty})$. Here $\map{\I}{(\Ac^{\infty})^{n}}{(\Ac^{\infty})^{n-1}}$ is the cochain contraction defined in Section \ref{subsec:CochainContractions}.

\medskip

In fact, with the definition of the connecting homomorphism $\Phi_{\L}^{n}$ in \eqref{map:ConnectingHomomorphismL} understood, this follows from the diagram
\[
\begin{tikzcd}[column sep = large, row sep = large]
	&(\Ac^{\infty})^{n-1} \arrow[r,"\L^{n-1}",twoheadrightarrow] \arrow[d,"\de^{n-1}",rightarrow]
	& \Ec^{n-1} \arrow[r,"i^{n-1}",hookrightarrow]
	& (\Ac^{\infty}_{\t})^{n-1}
	&
	& \\
	(\Ac_{P}^{\infty})^{n} \arrow[r,"\iota^{n}",hookrightarrow]
	& (\Ac^{\infty})^{n} \arrow[bend left, u,"\I^{n}",rightarrow]
	&
	&
	&
\end{tikzcd}
\]
together with the fact that $\I^{n}$ is a cochain contraction by Proposition \ref{prop:ComplexAIsAcyclic}.

\medskip

\noindent{\textbf{Step 6.}} We claim that the cocycle $u$ is given by the formula
\begin{equation} \label{eqn:FormulaForuTransgression}
	u = ( \L \, \I \, \Or) \cup b^{\prime\prime}.
\end{equation}

\medskip

To see this, we apply Lemma \ref{lemma:OperatorsActingOnCupProduct}\,(i, ii) to the defining formula for $u$ from Step 5. By \eqref{eqn:ProductCocycleb} we obtain
\[
u = \L \, \I \, b = \L \, \I \, ( \Or \cup b^{\prime\prime} ) = \L \, \bigl( (\I \, \Or) \cup b^{\prime\prime} \bigr) = ( \L \, \I \, \Or) \cup b^{\prime\prime} + ( \I \, \Or) \cup ( \L \, b^{\prime\prime} ).
\]
Since $b^{\prime\prime}$ is $P$-invariant, we have $\L b^{\prime\prime} = 0$. The claimed formula follows.

\medskip

\noindent{\textbf{Step 7.}} Define a cochain $v \in (\Ac^{\infty})^{n-1}$ by
\[
v \deq \I u.
\]
We claim that $v$ is a tame function. Since $\I$ is a cochain contraction by Proposition \ref{prop:ComplexAIsAcyclic}, this will then imply that $[u] = 0$ in $H^{n-1}(\Ac_{\t}^{\infty})$.

\medskip

First of all, by Lemma \ref{lemma:OperatorsActingOnCupProduct}\,(ii) we obtain from \eqref{eqn:FormulaForuTransgression} the expression
\[
v = ( \I \, \L \, \I \, \Or) \cup b^{\prime\prime}.
\]
By Lemma \ref{lemma:OperatorsActingOnCupProduct}\,(iii), the $K$-reduction of this cochain is the function
\[
v_{K} = ( \I \, \L \, \I \, \Or)_{K} \cup b^{\prime\prime},
\]
which by Lemma \ref{lemma:FunctionILIOr} and \eqref{eqn:KReduction} equals
\[
v_{K} = \frac{i}{\pi} \cdot b^{\prime\prime}.
\]
Since $b^{\prime\prime}$ is real valued, it follows that $\Re v_{K} = 0$. Moreover, by Lemma \ref{lemma:ReductionAndExtensionByK}\,(i) we may write the function $v$ as the $K$-extension $v = (v_{K})^{K}_{1}$. Hence it follows from Lemma \ref{lemma:CriterionForTameness} that $v$ is tame.

\medskip

\noindent{\textbf{Step 8.}} Combining the results from Step 3, Step 5 and Step 7, we have proved that
\[
i^{\ast} \circ (\Phi_{\L}^{n-1})^{-1} \, \beta = 0 \quad \text{and} \quad \Pi^{n} \, \beta = \alpha,
\]
which by Proposition \ref{prop:CriterionForTransgressive} implies that $\alpha$ is transgressive.
\end{proof}

\subsection{Proof of Theorem \ref{thm:VanishingForStronglyReducibleClasses} and Theorem \ref{thm:ImageOfTransgressionMap}}

By invariance of continuous bounded cohomology of connected Lie groups under local isomorphisms \cite[Cor.\,7.5.10]{Monod/Continuous-bounded-cohomology-of-locally-compact-groups} it will be enough to prove the theorem for the Lie group $G = \PU(1,1)$. Fix a degree $n>2$, and consider a strongly reducible class $\alpha \in H^{n}_{\cb}(G;\R)$. Then by Proposition \ref{prop:StronlgyReducibleImpliesTransgressive} the class $\alpha$ is transgressive, hence $\alpha = 0$ by Proposition \ref{prop:VanishingForTransgressiveClasses}.

\section{Construction of primitives}
\label{sec:ConstructionOfPrimitives}

\subsection{Explicit formulas for primitives}
\label{subsec:ExplicitFormulasForPrimitives}

Fix an integer $n>2$, and consider a $G$-invariant bounded cocycle $c \in A^{\infty}(\T^{n+1},\R)^{P}$ satisfying $\de c = 0$. By a \emph{primitive} of the cocycle $c$ we mean a $G$-invariant function $p \in A(\T^{n},\R)^{P}$ that solves the cohomological equation
\[
\de p = c.
\]
Note that we do not require the function $p$ to be bounded. The aim of this section is to provide a systematic way of constructing such primitives in explicit terms for any given $G$-invariant bounded cocycle $c$. We will moreover see that the primitives obtained in this way are bounded under suitable additional assumptions on the cocycle $c$.

\begin{proposition} \label{prop:ExplicitFormulasForPrimitives}
	Fix an integer $n>2$ and a measurable set of basepoints $B_{n} \subset \T^{n+1}$ for the boundary action of $G$ on $\T^{n+1}$. Let $c \in A^{\infty}(\T^{n+1},\R)^{P}$ be a $G$-invariant bounded cocycle satisfying $\de c = 0$. Define a function $p \in A(\T^{n},\R)$ by
\begin{equation} \label{eqn:PrimitivesDefinitionOfp}
	p \deq \I c - \de \Rop_{\Bc} u,
\end{equation}
where $u \in A^{\infty}(\T^{n-1},\C_{1})$ is the function
\begin{equation} \label{eqn:PrimitivesDefinitionOfu}
	u \deq \bigl( \Id - \de \Sop \I \Q \bigr) \I \L \I c
\end{equation}
(see Figure \ref{fig:Primitivep}). Here $\I$ is the cochain contraction in \eqref{eqn:DefinitionOfOperatorI}, $\L$ and $\Q$ are the differential operators in \eqref{map:OperatorLDownstairs} and \eqref{map:OperatorQDownstairs}, and $\Rop_{\Bc}$ and $\Sop$ are the integral operators in \eqref{map:SolutionOperatorCauchyProblem} and \eqref{map:SolutionOperatorFrobeniusProblem}.
\begin{enumerate}[leftmargin=1cm,topsep=0.5ex,itemsep=0.5ex]
	\item The function $p$ is a well-defined primitive for the cocycle $c$, i.e., $p \in A(\T^{n},\R)^{P}$ and
\[
\de p = c.
\]
	\item Assume in addition that the cocycle $c$ admits a product decomposition
\[
c = \Or \cup c^{\prime}
\]
for some cocycle $c^{\prime} \in A^{\infty}(\T^{n-1},\R)^{P}$, where $\Or \in A^{\infty}(\T^{3},\R)^{P}$ is the orientation cocycle defined in \eqref{eqn:OrientationCocycle} and $\cup$ denotes the cup product in \eqref{map:CupProductOnAinftyReal}. Then the primitive $p$ is bounded.
\end{enumerate}
\end{proposition}

Motivated by the schematic diagram in Figure \ref{fig:Primitivep}, we will refer to the formulas in \eqref{eqn:PrimitivesDefinitionOfp} and \eqref{eqn:PrimitivesDefinitionOfu} as the \emph{staircase construction} of the primitive $p$ for the cocycle $c$.

\begin{figure}[ht] \label{fig:Primitivep}
\begin{tikzcd}[column sep = large, row sep = large]
	&&
	A^{\infty}_{\t}(\T^{n-2},\C_{1}) \arrow[r,"\Q^{n-3}",twoheadrightarrow] \arrow[d,"\de^{n-3}",rightarrow]
	& A^{\infty}(\T^{n-2},\R) \arrow[d,"\de^{n-3}",rightarrow] \arrow[bend right, l,"\Sop^{n-3}"',rightarrow] \\
	& A(\T^{n-1},\R) \arrow[r,"\L^{n-2}",rightarrow] \arrow[d,"\de^{n-2}",rightarrow]
	& A^{\infty}(\T^{n-1},\C_{1}) \arrow[r,"\Q^{n-2}",twoheadrightarrow] \arrow[d,"\de^{n-2}",rightarrow] \arrow[bend right, l,"\Rop_{\Bc}^{n-2}"',dashrightarrow]
	& A^{\infty}(\T^{n-1},\R) \arrow[bend left, u,"\I^{n-2}",rightarrow] \\
	A(\T^{n},\R)^{P} \arrow[r,hookrightarrow] \arrow[d,"\de^{n-1}",rightarrow]
	& A(\T^{n},\R) \arrow[r,"\L^{n-1}",rightarrow] \arrow[d,"\de^{n-1}",rightarrow]
	& A^{\infty}(\T^{n},\C_{1}) \arrow[bend left, u,"\I^{n-1}",rightarrow]
	& \\
	A(\T^{n+1},\R)^{P} \arrow[r,hookrightarrow]
	& A(\T^{n+1},\R) \arrow[bend left, u,"\I^{n}",dashrightarrow]
	&&
\end{tikzcd}
\caption{The staircase construction of primitives. The dashed arrows indicate that the respective map is defined on a smaller domain.}
\end{figure}

\begin{proof}[Proof of Proposition \ref{prop:ExplicitFormulasForPrimitives}]
	Fix $n>2$ and a measurable set of basepoints $B_{n} \subset \T^{n+1}$ for the boundary action of $G$ on $\T^{n+1}$. Let $c \in A^{\infty}(\T^{n+1},\R)^{P}$ be such that $\de c = 0$.

\medskip

\noindent{\textbf{Step 1.}} We claim that the function $\I \L \I c \in A^{\infty}(\T^{n-1},\C_{1})$ satisfies
\[
\de \I \L \I c = \L \I c.
\]

\medskip

First of all, we note that by Proposition \ref{prop:ComplexAIsAcyclic} and Proposition \ref{prop:CauchyFrobeniusSequence} the function $\I \L \I c$ is well-defined. Observe that
\[
\de \L \I c = \L \de \I c = \L c = 0.
\]
Here in the first equality we used that $\L$ is a cochain map by Lemma \ref{lemma:OperatorsLAndQIntertwine}, while the second equality follows from Proposition \ref{prop:ComplexAIsAcyclic} since $c$ is a cocycle, and the third equality follows from Proposition \ref{prop:CauchyFrobeniusSequence} since $c$ is $P$-invariant. The claim is then a consequence of Proposition \ref{prop:ComplexAIsAcyclic}.

\medskip

\noindent{\textbf{Step 2.}} Define a function $u \in A^{\infty}(\T^{n-1},\C_{1})$ by
\begin{equation} \label{eqn:DefinitionOfu}
	u \deq \I \L \I c - \de \Sop \I \Q \I \L \I c = \bigl( \Id - \de \Sop \I \Q \bigr) \I \L \I c.
\end{equation}
It follows from Proposition \ref{prop:ComplexAIsAcyclic}, Proposition \ref{prop:CauchyFrobeniusSequence} and Proposition \ref{prop:FrobeniusProblem} that this function is well-defined.

\medskip

\noindent{\textbf{Step 3.}} We claim that $\de u = \L \I c$.

\medskip

This is immediate from \eqref{eqn:DefinitionOfu} using Step 1 and the fact that $\de ^{2} = 0$.

\medskip

\noindent{\textbf{Step 4.}} We claim that $\Q u = 0$.

\medskip

To prove this, we first observe that
\[
\de \Q \I \L \I c = \Q \de \I \L \I c = \Q \L \I c = 0
\]
by Step 1 and Proposition \ref{prop:CauchyFrobeniusSequence}. Hence it follows with Proposition \ref{prop:FrobeniusProblem} that
\[
\Q \de \Sop \I \Q \I \L \I c = \de \Q \Sop \I \Q \I \L \I c = \de \I \Q \I \L \I c = \Q \I \L \I c.
\]
The claim is now immediate from \eqref{eqn:DefinitionOfu}.

\medskip

\noindent{\textbf{Step 5.}} Define a function $p \in A(\T^{n},\R)$ by
\begin{equation} \label{eqn:DefinitionOfp}
	p \deq \I c - \de \Rop_{\Bc} u.
\end{equation}
It follows from Proposition \ref{prop:ComplexAIsAcyclic} and Proposition \ref{prop:CauchyProblem} in combination with Step~4 that this function is well-defined.

\medskip

\noindent{\textbf{Step 6.}} We claim that $\de p = c$.

\medskip

This follows from Proposition \ref{prop:ComplexAIsAcyclic} since $c$ is a cocycle, together with the fact that $\de ^{2} = 0$.

\medskip

\noindent{\textbf{Step 7.}} We claim that $\L p = 0$ and hence $p \in A(\T^{n+1},\R)^{P}$. Together with Step 6 this proves part (i) of the proposition.

\medskip

Using Proposition \ref{prop:CauchyProblem} and Step 3, we compute that
\[
\L \de \Rop_{\Bc} u = \de \L \Rop_{\Bc} u = \de u = \L \I c.
\]
The claim now follows from \eqref{eqn:DefinitionOfp}.

\medskip

\noindent{\textbf{Step 8.}} Assume from now that $c = \Or \cup c^{\prime}$ for some cocycle $c^{\prime} \in A^{\infty}(\T^{n-1},\R)^{P}$. We claim that the function $\I \L \I c \in A^{\infty}(\T^{n-1},\C_{1})$ is tame.

\medskip

Since $c^{\prime}$ is $P$-invariant and hence $\L c^{\prime} = 0$, we compute with Lemma \ref{lemma:OperatorsActingOnCupProduct}\,(i, ii) that
\[
\I \L \I c = \I \bigl( (\L \I \Or) \cup c^{\prime} + (\I \Or) \cup (\L c^{\prime}) \bigr) = (\I \L \I \Or) \cup c^{\prime}.
\]
By Lemma \ref{lemma:OperatorsActingOnCupProduct}\,(iii), the $K$-reduction of this function is
\[
(\I \L \I c)_{K} = ( \I \L \I \Or)_{K} \cup c^{\prime},
\]
which by Lemma \ref{lemma:FunctionILIOr} and \eqref{eqn:KReduction} equals
\[
(\I \L \I c)_{K} = \frac{i}{\pi} \cdot c^{\prime}.
\]
Since $c^{\prime}$ is real valued, it follows that $\Re (\I \L \I c)_{K} = 0$. Moreover, by Lemma \ref{lemma:ReductionAndExtensionByK}\,(i) we may write the function $\I \L \I c$ as the $K$-extension $\I \L \I c = ((\I \L \I c)_{K})^{K}_{1}$. Hence the claim follows from Lemma \ref{lemma:CriterionForTameness}.

\medskip

\noindent{\textbf{Step 9.}} We know from Proposition \ref{prop:FrobeniusProblem} that the function $\de \Sop \I \Q \I \L \I c \in A^{\infty}(\T^{n-1},\C_{1})$ is tame. Combining this with Step 8, it follows from \eqref{eqn:DefinitionOfu} that the function $u$ is tame.

\medskip

\noindent{\textbf{Step 10.}} Since $u$ is tame by Step 9, Proposition \ref{prop:CauchyProblem}\,(iii) implies that the function $\Rop_{\Bc} u$ in Step 5 is bounded. Since $c$ is bounded, the function $\I c$ is bounded by Proposition \ref{prop:ComplexAIsAcyclic}. Hence we conclude from \eqref{eqn:DefinitionOfp} that the primitive $p$ is bounded as well. This proves part (ii) of the proposition.
\end{proof}

\subsection{The operator $\P$}
\label{subsec:TheOperatorP}

Our goal in this section is to define the linear operator
\[
\map{\P^{n}}{L^{\infty}(\T^{n+1},\R)^{G} \supset \ker \de^{n}}{L^{0}(\T^{n},\R)^{G}} \quad\quad (n>2)
\]
that appears in Theorem \ref{thm:ExplicitFormulasForPrimitives}. To begin with, let us denote by
\[
\map{\piop^{n}}{A^{\infty}(\T^{n+1},\R)^{P}}{L^{\infty}(\T^{n+1},\R)^{G}} \quad\quad (n \ge 0)
\]
the natural cochain map \eqref{map:EquivariantLiftingFromA^P}. We have seen in the proof of Proposition \ref{prop:EquivariantLifting} that by a result of Monod, this map admits a section which we denote by
\[
\map{\sigmaop^{n}}{L^{\infty}(\T^{n+1},\R)^{G}}{A^{\infty}(\T^{n+1},\R)^{P}} \quad\quad (n \ge 0).
\]
Let us further fix a collection $\Bc = \{B_{n}\}_{n \ge 2}$ of measurable sets of basepoints $B_{n} \subset \T^{n+1}$ for the boundary action of $G$ on $\T^{n+1}$ for all $n \ge 2$ (cf.\,\cite[App.\,B]{Zimmer/Ergodic-theory-and-semisimple-groups}). We are now in a position to define the operator $\P$.

Let $n>2$, and let $c \in L^{\infty}(\T^{n+1},\R)^{G}$ be a $G$-invariant bounded function satisfying the cocycle relation $\de^{n} c = 0$. We then define
\begin{multline} \label{eqn:DefinitionOfOperatorP}
	\P^{n} c \deq \piop^{n-1} \I^{n} \sigmaop^{n} c \,\,- \\ \piop^{n-1} \de^{n-2} \Rop_{\Bc}^{n-2} \Big( \Id \,-\, \de^{n-3} \Sop^{n-3} \I^{n-2} \Q^{n-2} \Big) \I^{n-1} \L^{n-1} \I^{n} \sigmaop^{n} c. \hspace{15mm}
\end{multline}
Here $\I$ is the cochain contraction in \eqref{eqn:DefinitionOfOperatorI}, $\L$ and $\Q$ are the differential operators in \eqref{map:OperatorLDownstairs} and \eqref{map:OperatorQDownstairs}, and $\Rop_{\Bc}$ and $\Sop$ are the integral operators in \eqref{map:SolutionOperatorCauchyProblem} and \eqref{map:SolutionOperatorFrobeniusProblem}. Comparing with the formulas in \eqref{eqn:PrimitivesDefinitionOfp} and \eqref{eqn:PrimitivesDefinitionOfu}, it follows from Proposition \ref{prop:ExplicitFormulasForPrimitives} that the function $\P^{n} c$ is in fact well-defined. 

The formula in \eqref{eqn:DefinitionOfOperatorP} is illustrated schematically in Figure \ref{fig:Primitivep}. By abuse of notation, we will usually suppress the maps $\pi$ and $\sigma$, writing
\[
\P c = \I c - \de \Rop_{\Bc} \bigl( \Id - \de \Sop \I \Q \bigr) \I \L \I c
\]
for short. One should, however, keep in mind that the right-hand side of this formula will only be defined for representatives of the cocycle $c$ that are contained in the space $A^{\infty}(\T^{n+1},\R)^{P}$.

\subsection{Proof of Theorem \ref{thm:ExplicitFormulasForPrimitives}}

Let $\P$ be the linear operator defined by \eqref{eqn:DefinitionOfOperatorP} in Section \ref{subsec:TheOperatorP}. It is well-defined by Proposition \ref{prop:ExplicitFormulasForPrimitives}. This proves (i). Comparing with \eqref{eqn:PrimitivesDefinitionOfp} and \eqref{eqn:PrimitivesDefinitionOfu}, we see that (ii) and (iii) follow from the corresponding statements in Proposition \ref{prop:ExplicitFormulasForPrimitives}\,(i, ii).


\begin{thebibliography}{10}
\bibliographystyle{amsplain}

\bibitem{AntolinMjSistoTaylor/Intersection-properties-of-stable-subgroups-and-bounded-cohomology}
Y.~Antolin, M.~Mj, A.~Sisto, and S.~J. Taylor, \emph{Intersection properties of
  stable subgroups and bounded cohomology}, {\tt arXiv:1612.07227 [math.GT]}
  (2016).

\bibitem{Arnold/Ordinary-differential-equations}
V.~I. Arnol'd, \emph{Ordinary differential equations}, Springer Textbook,
  Springer-Verlag, Berlin, 1992.

\bibitem{BerhanuCordaroHounie/An-introduction-to-involutive-structures}
S.~Berhanu, P.~D. Cordaro, and J.~Hounie, \emph{An introduction to involutive
  structures}, New Mathematical Monographs, vol.~6, Cambridge University Press,
  Cambridge, 2008.

\bibitem{BestvinaBrombergFujiwara/Bounded-cohomology-with-coefficients-in-uniformly-convex-Banach-spaces}
M.~Bestvina, K.~Bromberg, and K.~Fujiwara, \emph{Bounded cohomology with
  coefficients in uniformly convex {B}anach spaces}, Comment. Math. Helv.
  \textbf{91} (2016), no.~2, 203--218.

\bibitem{Bestvina/Bounded-cohomology-of-subgroups-of-mapping-class-groups}
M.~Bestvina and K.~Fujiwara, \emph{Bounded cohomology of subgroups of mapping
  class groups}, Geom. Topol. \textbf{6} (2002), 69--89.

\bibitem{Bloch/Higher-regulators-algebraic-K-theory-and-zeta-functions-of-elliptic-curves}
S.~J. Bloch, \emph{Higher regulators, algebraic {$K$}-theory, and zeta
  functions of elliptic curves}, CRM Monograph Series, vol.~11, American
  Mathematical Society, Providence, RI, 2000.

\bibitem{Brooks/Some-remarks-on-bounded-cohomology}
R.~Brooks, \emph{Some remarks on bounded cohomology}, Riemann surfaces and
  related topics: {P}roceedings of the 1978 {S}tony {B}rook {C}onference
  ({S}tate {U}niv. {N}ew {Y}ork, {S}tony {B}rook, {N}.{Y}., 1978), Ann. of
  Math. Stud., vol.~97, Princeton Univ. Press, Princeton, N.J., 1981,
  pp.~53--63.

\bibitem{Bucher/The-norm-of-the-Euler-class}
M.~Bucher and N.~Monod, \emph{The norm of the {E}uler class}, Math. Ann.
  \textbf{353} (2012), no.~2, 523--544.

\bibitem{BucherMonod/The-bounded-cohomology-of-SL2-over-local-fields-and-S-integers}
\bysame, \emph{The bounded cohomology of ${SL}_2$ over local fields and
  {S}-integers}, {\tt arXiv:1609.05013 [math.GR]}, to appear in IMRN (2016).

\bibitem{BucherMonod/The-cup-product-of-Brooks-quasimorphisms}
\bysame, \emph{The cup product of {B}rooks quasimorphisms}, Forum Math.
  \textbf{30} (2018), no.~5, 1157--1162.

\bibitem{Bucher-Karlsson/Finiteness-properties-of-characteristic-classes-of-flat-bundles}
M.~Bucher-Karlsson, \emph{Finiteness properties of characteristic classes of
  flat bundles}, Enseign. Math. (2) \textbf{53} (2007), no.~1-2, 33--66.

\bibitem{Bucher-Karlsson/Simplicial-volume-of-locally-symmetric-spaces-covered-by-rm-SLsb-3Bbb-R/rm-SO3}
\bysame, \emph{Simplicial volume of locally symmetric spaces covered by {${\rm
  SL}\sb 3\Bbb R/{\rm SO}(3)$}}, Geom. Dedicata \textbf{125} (2007), 203--224.

\bibitem{Bucher-Karlsson/The-simplicial-volume-of-closed-manifolds-covered-by-Bbb-Hsp-2timesBbb-Hsp-2}
\bysame, \emph{The simplicial volume of closed manifolds covered by {$\Bbb H\sp
  2\times\Bbb H\sp 2$}}, J. Topol. \textbf{1} (2008), no.~3, 584--602.

\bibitem{BurgerIozzi/A-useful-formula-from-bounded-cohomology}
M.~Burger and A.~Iozzi, \emph{A useful formula from bounded cohomology},
  G\'eom\'etries \`a courbure n\'egative ou nulle, groupes discrets et
  rigidit\'es, S\'emin. Congr., vol.~18, Soc. Math. France, Paris, 2009,
  pp.~243--292.

\bibitem{Burger/Bounded-cohomology-of-lattices-in-higher-rank-Lie-groups}
M.~Burger and N.~Monod, \emph{Bounded cohomology of lattices in higher rank
  {L}ie groups}, J. Eur. Math. Soc. (JEMS) \textbf{1} (1999), no.~2, 199--235.

\bibitem{Burger/Continuous-bounded-cohomology-and-applications-to-rigidity-theory}
\bysame, \emph{Continuous bounded cohomology and applications to rigidity
  theory}, Geom. Funct. Anal. \textbf{12} (2002), no.~2, 219--280.

\bibitem{Burger/On-and-around-the-bounded-cohomology-of-SL2}
\bysame, \emph{On and around the bounded cohomology of {$SL_2$}}, Rigidity in
  {D}ynamics and {G}eometry, Springer, 2002, pp.~19--37.

\bibitem{CandelConlon/Foliations.-I}
A.~Candel and L.~Conlon, \emph{Foliations. {I}}, Graduate Studies in
  Mathematics, vol.~23, American Mathematical Society, Providence, RI, 2000.

\bibitem{Dupont/Bounds-for-characteristic-numbers-of-flat-bundles}
J.~L. Dupont, \emph{Bounds for characteristic numbers of flat bundles},
  Algebraic topology, {A}arhus 1978 ({P}roc. {S}ympos., {U}niv. {A}arhus,
  {A}arhus, 1978), Lecture Notes in Math., vol. 763, Springer, 1979.

\bibitem{Epstein/The-second-bounded-cohomology-of-word-hyperbolic-groups}
D.~B.~A. Epstein and K.~Fujiwara, \emph{The second bounded cohomology of
  word-hyperbolic groups}, Topology \textbf{36} (1997), no.~6, 1275--1289.

\bibitem{Farre/Bounded-cohomology-of-finitely-generated-Kleinian-groups}
J.~Farre, \emph{Bounded cohomology of finitely generated {K}leinian groups},
  {\tt arXiv:1706.02001 [math.GT]} (2017).

\bibitem{Farre/Relations-in-bounded-cohomology}
\bysame, \emph{Relations in bounded cohomology}, {\tt arXiv:1808.05711
  [math.GT]} (2018).

\bibitem{FranceschiniFrigerioPozzettiSisto/The-zero-norm-subspace-of-bounded-cohomology-of-acylindrically-hyperbolic-groups}
F.~Franceschini, R.~Frigerio, B.~Pozzetti, and A.~Sisto, \emph{The zero norm
  subspace of bounded cohomology of acylindrically hyperbolic groups}, {\tt
  arXiv:1703.03752 [math.GR]} (2017).

\bibitem{Goncharov/Geometry-of-configurations-polylogarithms-and-motivic-cohomology}
A.~B. Goncharov, \emph{Geometry of configurations, polylogarithms, and motivic
  cohomology}, Adv. Math. \textbf{114} (1995), no.~2, 197--318.

\bibitem{Grigorchuk/Some-results-on-bounded-cohomology}
R.~I. Grigorchuk, \emph{Some results on bounded cohomology}, Combinatorial and
  geometric group theory ({E}dinburgh, 1993), London Math. Soc. Lecture Note
  Ser., vol. 204, Cambridge Univ. Press, Cambridge, 1995, pp.~111--163.

\bibitem{Gromov/Volume-and-bounded-cohomology}
M.~Gromov, \emph{Volume and bounded cohomology}, Inst. Hautes \'Etudes Sci.
  Publ. Math. (1982), no.~56, 5--99 (1983).

\bibitem{Hartnick/Surjectivity-of-the-comparison-map-in-bounded-cohomology-for-Hermitian-Lie-groups}
T.~Hartnick and A.~Ott, \emph{Surjectivity of the comparison map in bounded
  cohomology for {H}ermitian {L}ie groups}, Int. Math. Res. Not. IMRN (2012),
  no.~9, 2068--2093.

\bibitem{Hartnick/Bounded-cohomology-via-partial-differential-equations-I}
\bysame, \emph{Bounded cohomology via partial differential equations, {I}},
  Geom. Topol. \textbf{19} (2015), no.~6, 3603--3643.

\bibitem{HartnickOtt/Perturbations-of-the-Spence-Abel-equation-and-deformations-of-the-dilogarithm-function}
\bysame, \emph{Perturbations of the {S}pence-{A}bel equation and deformations
  of the dilogarithm function}, Math. Ann. \textbf{368} (2017), no.~3-4,
  1399--1428.

\bibitem{HartnickSisto/Bounded-cohomology-and-virtually-free-hyperbolically-embedded-subgroups}
T.~Hartnick and A.~Sisto, \emph{Bounded cohomology and virtually free
  hyperbolically embedded subgroups}, {\tt arXiv:1701.00686 [math.GR]} (2017).

\bibitem{Heuer/Cup-Product-in-Bounded-Cohomology-of-the-Free-Group}
N.~Heuer, \emph{Cup product in bounded cohomology of the free group}, { \tt
  arXiv:1710.03193 [math.GR]} (2017).

\bibitem{HullOsin/Induced-quasicocycles-on-groups-with-hyperbolically-embedded-subgroups}
M.~Hull and D.~Osin, \emph{Induced quasicocycles on groups with hyperbolically
  embedded subgroups}, Algebr. Geom. Topol. \textbf{13} (2013), no.~5,
  2635--2665.

\bibitem{Ivanov/Foundations-of-the-theory-of-bounded-cohomology}
N.~V. Ivanov, \emph{Foundations of the theory of bounded cohomology}, Zap.
  Nauchn. Sem. Leningrad. Otdel. Mat. Inst. Steklov. (LOMI) \textbf{143}
  (1985), 69--109, 177--178, Studies in topology, V.

\bibitem{Johnson/Cohomology-in-Banach-algebras}
B.~E. Johnson, \emph{Cohomology in {B}anach algebras}, American Mathematical
  Society, Providence, R.I., 1972, Memoirs of the American Mathematical
  Society, No. 127.

\bibitem{Kerby/On-infinite-sharply-multiply-transitive-groups}
W.~Kerby, \emph{On infinite sharply multiply transitive groups},
  Vandenhoeck\thinspace \&\thinspace Ruprecht, G\"ottingen, 1974, Hamburger
  Mathematische Einzelschriften, Neue Folge, Heft 6.

\bibitem{Knapp/Lie-groups-beyond-an-introduction}
A.~W. Knapp, \emph{Lie groups beyond an introduction}, second ed., Progress in
  Mathematics, vol. 140, Birkh\"auser Boston Inc., Boston, MA, 2002.

\bibitem{Lafont/Simplicial-volume-of-closed-locally-symmetric-spaces-of-non-compact-type}
J.-F. Lafont and B.~Schmidt, \emph{Simplicial volume of closed locally
  symmetric spaces of non-compact type}, Acta Math. \textbf{197} (2006), no.~1,
  129--143.

\bibitem{Loh/A-note-on-bounded-cohomological-dimension-of-discrete-groups}
C.~L\"{o}h, \emph{A note on bounded-cohomological dimension of discrete
  groups}, J. Math. Soc. Japan \textbf{69} (2017), no.~2, 715--734.

\bibitem{Mineyev/Bounded-cohomology-characterizes-hyperbolic-groups}
I.~Mineyev, \emph{Bounded cohomology characterizes hyperbolic groups}, Q. J.
  Math. \textbf{53} (2002), no.~1, 59--73.

\bibitem{Monod/Continuous-bounded-cohomology-of-locally-compact-groups}
N.~Monod, \emph{Continuous bounded cohomology of locally compact groups},
  Lecture Notes in Mathematics, vol. 1758, Springer-Verlag, Berlin, 2001.

\bibitem{Monod/An-invitation-to-bounded-cohomology}
\bysame, \emph{An invitation to bounded cohomology}, International {C}ongress
  of {M}athematicians. {V}ol. {II}, Eur. Math. Soc., Z\"urich, 2006,
  pp.~1183--1211.

\bibitem{Monod/Vanishing-up-to-the-rank-in-bounded-cohomology}
\bysame, \emph{Vanishing up to the rank in bounded cohomology}, Math. Res.
  Lett. \textbf{14} (2007), no.~4, 681--687.

\bibitem{Monod/On-the-bounded-cohomology-of-semi-simple-groups-S-arithmetic-groups-and-products}
\bysame, \emph{On the bounded cohomology of semi-simple groups,
  {$S$}-arithmetic groups and products}, J. Reine Angew. Math. \textbf{640}
  (2010), 167--202.

\bibitem{Monod/Equivariant-measurable-liftings}
\bysame, \emph{Equivariant measurable liftings}, Fund. Math. \textbf{230}
  (2015), no.~2, 149--165. \MR{3337223}

\bibitem{Pieters/Continuous-cohomology-of-the-isometry-group-of-hyperbolic-space-realizable-on-the-boundary}
H.~Pieters, \emph{{C}ontinuous cohomology of the isometry group of hyperbolic
  space realizable on the boundary}, {\tt arXiv:1507.04915} (2015).

\bibitem{Soma/Bounded-cohomology-and-topologically-tame-Kleinian-groups}
T.~Soma, \emph{Bounded cohomology and topologically tame {K}leinian groups},
  Duke Math. J. \textbf{88} (1997), no.~2, 357--370.

\bibitem{Soma/Bounded-cohomology-of-closed-surfaces}
\bysame, \emph{Bounded cohomology of closed surfaces}, Topology \textbf{36}
  (1997), no.~6, 1221--1246.

\bibitem{Soma/The-zero-norm-subspace-of-bounded-cohomology}
\bysame, \emph{The zero-norm subspace of bounded cohomology}, Comment. Math.
  Helv. \textbf{72} (1997), no.~4, 582--592.

\bibitem{Sugiura/Unitary-representations-and-harmonic-analysis}
M.~Sugiura, \emph{Unitary representations and harmonic analysis}, second ed.,
  North-Holland Mathematical Library, vol.~44, North-Holland Publishing Co.,
  Amsterdam; Kodansha, Ltd., Tokyo, 1990, An introduction.

\bibitem{Thurston/Three-dimensional-geometry-and-topology.-Vol.-1}
W.~P. Thurston, \emph{Three-dimensional geometry and topology. {V}ol. 1},
  Princeton Mathematical Series, vol.~35, Princeton University Press,
  Princeton, NJ, 1997, Edited by Silvio Levy.

\bibitem{Zimmer/Ergodic-theory-and-semisimple-groups}
R.~J. Zimmer, \emph{Ergodic theory and semisimple groups}, Monographs in
  Mathematics, vol.~81, Birkh\"{a}user Verlag, Basel, 1984.

\end{thebibliography}
\end{document}